\documentclass[gdplain]{geradwp}

\PassOptionsToPackage{hyphens}{url}

\usepackage[english]{babel}
\usepackage[square,numbers,comma]{natbib}
\usepackage{lastpage}
\usepackage{xspace}
\usepackage{fancyhdr}
\usepackage{geometry}
\usepackage{setspace}
\usepackage{array}
\usepackage{multirow}
\usepackage[utf8]{inputenc}
\usepackage[T1]{fontenc}
\usepackage{lmodern}
\usepackage[dvipsnames]{xcolor}
\usepackage{hyperref}
\usepackage{caption}
\usepackage{enumitem}
\usepackage{graphicx}
\usepackage{float}
\usepackage{amsmath}
\usepackage{amsthm}
\usepackage{dsfont}
\usepackage{amssymb}
\usepackage{stmaryrd}
\usepackage{mathtools}
\usepackage[ruled]{algorithm}
\usepackage[noend]{algpseudocode}
\usepackage{textcomp}
\usepackage{gensymb}
\usepackage{eurosym}
\usepackage[backgroundcolor=yellow,linecolor=orange,textsize=scriptsize]{todonotes}

\GDcoverpagewhitespace{6.8cm}
\hypersetup{colorlinks,%
citecolor={blue}, 
urlcolor={blue},
linkcolor={blue},
breaklinks={true}
}

\makeatletter
\ifthenelse{\isundefined{\ALG@name}}{}%
{%
\renewcommand{\ALG@name}{\sffamily\footnotesize Algorithm}
}
\makeatother
\ifthenelse{\isundefined{\SetAlCapNameFnt}}{}%
{%
\SetAlCapNameFnt{\footnotesize}
\SetAlCapFnt{\sffamily\footnotesize}
}

\newcommand{\comment}[1]{}
\newcommand{\refbis}[2]{\ref{#1}.\ref{#1:#2}}
\newcolumntype{R}[1]{>{\raggedleft\arraybackslash}p{#1}}
\newcolumntype{M}[1]{>{\centering\arraybackslash}m{#1}}

\newcommand{\Acal}{\mathcal{A}}  
\newcommand{\Bcal}{\mathcal{B}}  
  
\newcommand{\Dcal}{\mathcal{D}}  
\newcommand{\Ecal}{\mathcal{E}}

\newcommand{\Hcal}{\mathcal{H}}

\newcommand{\Kcal}{\mathcal{K}}  
  
\newcommand{\Mcal}{\mathcal{M}}  
\newcommand{\Ncal}{\mathcal{N}}  \newcommand{\Nbb}{\mathbb{N}}
\newcommand{\Ocal}{\mathcal{O}}  \newcommand{\Obb}{\mathbb{O}}
 \newcommand{\Pbf}{\mathbf{P}} 
  \newcommand{\Qbb}{\mathbb{Q}}
\newcommand{\Rcal}{\mathcal{R}}  \newcommand{\Rbb}{\mathbb{R}}
\newcommand{\Scal}{\mathcal{S}}  \newcommand{\Sbb}{\mathbb{S}}
\newcommand{\Tcal}{\mathcal{T}}  
\newcommand{\Ucal}{\mathcal{U}}

\newcommand{\Xcal}{\mathcal{X}}  
  
  \newcommand{\Zbb}{\mathbb{Z}}

\newcommand{\algoname}[1]{\texttt{#1}\xspace}
\newcommand{\dfo}{\algoname{DFO}}
\newcommand{\dsm}{\algoname{DSM}}
\newcommand{\rdsm}{\algoname{rDSM}}
\newcommand{\cdsm}{\algoname{cDSM}}
\newcommand{\discomads}{\algoname{DiscoMads}}
\def\poll{\algoname{poll}}
\def\covering{\algoname{covering}}
\def\revealing{\algoname{revealing}}
\def\search{\algoname{search}}
\def\update{\algoname{update}}

\algrenewcommand{\textproc}{\textit}
\algrenewcommand{\Return}{\State\algorithmicreturn\xspace}
\algnewcommand{\Not}[1]{\textbf{not}(#1)}
\algnewcommand{\And}{\textbf{and}}

\newcommand{\defequal}{\triangleq}
\renewcommand{\int}{\integ}
\DeclareMathOperator*{\argmin}{argmin}
\DeclareMathOperator*{\argmax}{argmax}
\DeclareMathOperator*{\minimize}{minimize}
\newcommand{\1}{\mathds{1}}
\newcommand{\norm}[1]{\left\lVert{#1}\right\rVert}
\newcommand{\textabs }[1]{\lvert{#1}\rvert}
\newcommand{\textnorm}[1]{\lVert{#1}\rVert}
\newcommand{\transpose}[1]{#1^\top}
\def\ll{\llbracket}
\def\rr{\rrbracket}
\renewcommand{\int}{\mathrm{int}}
\newcommand{\cl}{\mathrm{cl}}
\def\partitioncup{\sqcup}
\newcommand{\round}{\mathrm{round}}
\newcommand{\dist}{\mathrm{dist}}
\newcommand{\pointplusset}[2]{\{#1\}+#2}
\newcommand{\DcalS}{\Dcal_{\tt{S}}} \newcommand{\TcalS}{\Tcal_{\tt{S}}} \newcommand{\tS}{t_{\tt{S}}} 
\newcommand{\DcalC}{\Dcal_{\tt{C}}} \newcommand{\TcalC}{\Tcal_{\tt{C}}} \newcommand{\tC}{t_{\tt{C}}} \newcommand{\dC}{d_{\tt{C}}}
\newcommand{\DcalP}{\Dcal_{\tt{P}}} \newcommand{\TcalP}{\Tcal_{\tt{P}}} \newcommand{\tP}{t_{\tt{P}}} 
 \newcommand{\TcalR}{\Tcal_{\tt{R}}}  
\newcommand{\DcalO}{\Dcal_{\tt{O}}} \newcommand{\TcalO}{\Tcal_{\tt{O}}} \newcommand{\tO}{t_{\tt{O}}} 

\newcommand{\acc}{\Acal}

\newtheorem{assumption}{Assumption}
\newtheorem{proposition}{Proposition}

\newtheorem{definition}{Definition}
\newtheorem{theorem}{Theorem}
\newtheorem{remark}{Remark}
\newtheorem{example}{Example}

\newtheorem{assertion}{Assertion}
\newtheorem{property}{Property}

\newcommand{\fct}[5]{
    #1:\left\{\begin{array}{ccl}
        #3 & \to     & #5\\
        #2 & \mapsto & #4
    \end{array}\right.}

\newcommand{\compactarray}[1]{\begingroup\renewcommand{\arraystretch}{1.0}\begin{array}{c}#1\end{array}\endgroup}
\newcommand{\problemoptimfree}[3]{\underset{\compactarray{#2}}{#1} \quad #3}

\newcommand{\overbracebelow}[2]{\underset{{\overbrace{\begin{subarray}{r}#1\end{subarray}}}}{#2}}

\ExplSyntaxOn
\NewDocumentCommand{\subarrayfill}{mm}{\str_if_in:nnT { #1 } { #2 } { \hfil }}
\ExplSyntaxOff
\makeatletter
\renewenvironment{subarray}[1]{
  \vcenter\bgroup
  \Let@ \restore@math@cr \default@tag
  \baselineskip\fontdimen10 \scriptfont\tw@
  \advance\baselineskip\fontdimen12 \scriptfont\tw@
  \lineskip\thr@@\fontdimen8 \scriptfont\thr@@
  \lineskiplimit\lineskip
  \ialign\bgroup
    \subarrayfill{cr}{#1}
    $\m@th\scriptstyle##$%
    \subarrayfill{cl}{#1}
    \crcr
}{\crcr\egroup\egroup}
\makeatother

\setlist[enumerate]{noitemsep, topsep=-\parskip}
\GDtitle{Convergence towards a local minimum by direct search methods with a covering step}
\GDmonth{Janvier}{January}
\GDyear{2024}
\GDnumber{TBD}
\GDauthorsShort{B. Hamel, K. H\'ebert}
\GDauthorsCopyright{Hamel, H\'ebert}
\GDpostpubcitation{Hamel, Benoit, Karine H\'ebert (2021). \og Un exemple de citation \fg, \emph{Journal of Journals}, vol. X issue Y, p. n-m}{https://www.gerad.ca/fr}
\GDrevised{Mai}{May}{2022}

\begin{document}


\begin{GDtitlepage}

\begin{GDauthlist}
\GDauthitem{Pierre-Yves Bouchet \ref{affil:polymtl}\GDrefsep\ref{affil:gerad}}
\GDauthitem{Charles Audet \ref{affil:polymtl}\GDrefsep\ref{affil:gerad}}
\GDauthitem{Loïc Bourdin \ref{affil:unilim}}
\end{GDauthlist}

\begin{GDaffillist}
\GDaffilitem{affil:polymtl}{École Polytechnique de Montréal, Montr\'eal (Qc), Canada, H3T 1J4}
\GDaffilitem{affil:gerad}{GERAD, Montr\'eal (Qc), Canada, H3T 1J4}
\GDaffilitem{affil:unilim}{XLIM Research Institute, University of Limoges, France}
\end{GDaffillist}

\begin{GDemaillist}
\GDemailitem{pierre-yves.bouchet@polymtl.ca}
\GDemailitem{charles.audet@gerad.ca}
\GDemailitem{loic.bourdin@unilim.fr}
\end{GDemaillist}

\end{GDtitlepage}


\GDabstracts

\begin{GDabstract}{Abstract}
This paper introduces a new step to the \textit{Direct Search Method} (\dsm) to strengthen its convergence analysis.
By design, this so-called \textit{\covering step} may ensure that, for all refined points of the sequence of incumbent solutions generated by the resulting \cdsm (\textit{\covering \dsm}), the set of all evaluated trial points is dense in a neighborhood of that refined point.
We prove that this additional property guarantees that all refined points are local solutions to the optimization problem.
This new result holds true even for a discontinuous objective function, under a mild assumption that we discuss in details.
We also provide a practical construction scheme for the \covering step that works at low additional cost per iteration.
Finally, we show that the \covering step may be adapted to classes of algorithms differing from the \dsm.
\end{GDabstract}

\begin{GDacknowledgements}
We thank Dr.\ Solène Kojtych for her insightful comments when we were investigating the properties of the \revealing step.
She deserves credit for Theorem~\ref{theorem:convergence_cdsm} since she also foresaw that the \revealing step could strengthen the convergence analysis of \dsm.
This work was funded by Audet's NSERC Canada Discovery Grant 2020--04448.
\end{GDacknowledgements}


\clearpage\newpage
\GDarticlestart

\section{Introduction}
\label{section:intro}

Consider the optimization problem
\begin{equation}
    \label{problem:P}
    \tag{$\Pbf$}
    \problemoptimfree{\minimize}{x \in \Rbb^n}{f(x)},
\end{equation}
where~$f: \Rbb^n \to \Rbb\cup\{+\infty\}$ is possibly discontinuous.
Few \textit{derivative-free optimization} (\dfo) methods~\cite{AuHa2017,CoScVibook,KvSe20GLobalOptim,SeKv17GlobalOptim,StPa23BookDIRECT} are studied in this context.
This paper focuses on the \textit{Direct Search Method}~(\dsm)~\cite[Part~3]{AuHa2017}, which is supported by a thorough convergence analysis in the presence of discontinuities under light assumptions about~$f$~\cite{AuBoBo22,ViCu2012}; although others methods exist, such as~\cite{CoMo98PWDisc,WeBa14GlobalDisc} which assume additional structure on~$f$ or~\cite{BaUs21OneDimDisc,StSe13GlobalOptim} relying on \textit{space-filling curves}~\cite{Ha60} to reduce the problem to a one-dimensional expression.
The~\dsm addresses Problem~\eqref{problem:P} by generating sequences~$(x^k)_{k \in \Nbb}$ of incumbent solutions and~$(\delta^k)_{k \in \Nbb}$ of poll radii, where~$(x^k,\delta^k) \in \Rbb^n \times \Rbb_+^*$ for all~$k \in \Nbb$.
The literature about \dsm extracts \textit{refining subsequences} from~$(x^k,\delta^k)_{k \in \Nbb}$ and studies their associated \textit{refined points}.
A subsequence~$(x^k,\delta^k)_{k \in K^*}$, where~$K^* \subseteq \Nbb$ is infinite, is said to be \textit{refining}~\cite{AuDe03a,AuDe2006} if all iterations indexed by~$k \in K^*$ \textit{fail} (that is,~$x^{k+1} = x^k$),~$(\delta^k)_{k \in K^*}$ converges to zero and~$(x^k)_{k \in K^*}$ converges to a limit~$x^*$ named the \textit{refined point}.
It is proved in~\cite{ViCu2012} that, for all sets~$K^* \subseteq \Nbb$ indexing a refining subsequence, the corresponding refined point~$x^*$ satisfies a necessary optimality condition expressed in term of the Rockafellar derivative~\cite{Rock80a,ViCu2012} of~$f$ at~$x^*$, provided that~$(f(x^k))_{k \in K^*}$ converges to~$f(x^*)$.
Then, our previous work~\cite{AuBoBo22} extends~\cite{ViCu2012} to ensure this last requirement under the assumption that all refining subsequences admit the same refined point.
These results are valid for two standard classes of \dsm: the \textit{mesh}-based method~\cite[Part~3]{AuHa2017} and the \textit{sufficient decrease}-based method~\cite[Section~7.7]{CoScVibook}.

This paper has two main goals and two related auxiliary ones.
The first main goal is to strengthen the above convergence analysis, via the addition of a new step to the \dsm to ensure that all refined points are local solutions.
For this purpose, we introduce the \covering \dsm (\cdsm), relying on the so-called \covering step which aims to ensure that the set of all evaluated trial points is dense in a neighborhood of any refined point.
We refer to Property~\ref{property:DCP} for details.
Then, we prove that Property~\ref{property:DCP} implies that, for all~$K^* \subseteq \Nbb$ indexing a refining subsequence with refined point denoted by~$x^*$, either~$x^* \in X \defequal \{x \in \Rbb^n: f(x) \neq +\infty\}$ is a local solution to Problem~\eqref{problem:P}, or~$x^* \notin X$ is such that~$(f(x^k))_{k \in K^*}$ converges to the infimum of~$f$ over a neighborhood of~$x^*$.
This result is formalized in Theorem~\ref{theorem:convergence_cdsm}.
The related auxiliary goal is to propose a practical construction scheme for the \covering step ensuring Property~\ref{property:DCP} by design.
This scheme fits in both the mesh-based \cdsm and the sufficient decrease-based \cdsm, and the additional cost per iteration it induces is low.
Our second main goal is to show that the \covering step is compatible with many methods, to allow for a convergence analysis close to Theorem~\ref{theorem:convergence_cdsm}.
Theorem~\ref{theorem:convergence_algo_with_covering} formalizes this study.
Our related auxiliary goal is to prove that our assumption about Problem~\eqref{problem:P}, involved in Theorems~\ref{theorem:convergence_cdsm} and~\ref{theorem:convergence_algo_with_covering}, cannot be relaxed in general and is weaker than in former work.
This paper originates from the corresponding author's PhD thesis~\cite[Chapter~4]{BouchetPhD} (in French).

Note that the \covering step generalizes the \revealing step from~\cite{AuBaKo22,AuBoBo22}.
Actually, our initial motivation was to better study this \revealing step.
Its goal in~\cite{AuBaKo22} is to reveal local discontinuities, but we observed in~\cite{AuBoBo22} that, when the so-called \revealing \dsm (\rdsm) generates a unique refined point, the \revealing step provides the density of the trial points in a neighborhood of the refined point.
However,~\cite{AuBaKo22,AuBoBo22} fail to deduce the local optimality of the refined point.
Thus, the current work originally aimed to state this property.
Yet, we eventually found that the \revealing step admits a generalization providing the density of the set of trial points around an arbitrary number of refined points.
The formalization of this generalization and the study of its properties constitute the core of the present work.
Note that we decided to change the terminology because, in comparison with the name \revealing, we believe that the name \covering better captures what this step actually does.

This paper is organized as follows.
Section~\ref{section:main_result_1} formalizes the \cdsm and states Theorem~\ref{theorem:convergence_cdsm}.
Section~\ref{section:proof} proves Theorem~\ref{theorem:convergence_cdsm}.
Section~\ref{section:comments_covering_step} provides a construction scheme for the \covering step.
Section~\ref{section:main_result_2} states and proves Theorem~\ref{theorem:convergence_algo_with_covering}.
Section~\ref{section:comments_assumption_partition} discusses our assumption about Problem~\eqref{problem:P}.
Section~\ref{section:general_comments} discusses our work and its possible extensions.
In addition, Appendix~\ref{appendix:materials} contains supplementary materials and Appendix~\ref{appendix:propositions} contains proofs of some auxiliary results.

\noindent {\bf Notation:}
We denote by~$\Sbb^n$ the unit sphere of~$\Rbb^n$.
For all~$r \in \Rbb_+^*$ and all~$x \in \Rbb^n$, we denote by~$\Bcal_r(x)$ the open ball of radius~$r$ centered at~$x$, and by~$\Bcal_r \defequal \Bcal_r(0)$.
For all~$x \in \Rbb^n$ and all~$\Scal \subseteq \Rbb^n$, we denote by~$\pointplusset{x}{\Scal} \defequal \{x+s : s \in \Scal\}$,~$f(\Scal) \defequal \{f(s) : s \in \Scal\}$ and~$\dist(x,\Scal) \defequal \inf\{\textnorm{x-s} : s \in \Scal\}$.
For all~$\Scal \subseteq \Rbb^n$,~$\Scal$ is said to be \textit{ample} if~$\Scal \subseteq \cl(\int(\Scal))$, or \textit{locally thin} if there exists~$\Ncal \subseteq \Rbb^n$ open so that~$\Scal \cap \Ncal \neq \emptyset = \int(\Scal) \cap \Ncal$.
Note that a set is either ample or locally thin (see Proposition~\ref{proposition:not_ample_equiv_locally_thin}), and that the definition of an ample set is equivalent to that of a \textit{semi-open} set introduced in~\cite{Le63SemiOpen}.
For all~$x \in \Rbb^n$ and all~$\Scal \subseteq \Rbb^n$ nonempty and closed, we denote by~$\round(x,\Scal) \defequal \argmin \dist(x,\Scal)$.
For all collections~$(\Scal_i)_{i=1}^{N}$ of subsets of~$\Rbb^n$ (with~$N \in \Nbb^* \cup \{+\infty\}$), their union is denoted by~$\partitioncup_{i=1}^{N} \Scal_i$ when the sets are pairwise disjoint.
Finally, for all~$(a,b) \in \Zbb^2$, we denote by~$\ll a,b\rr \defequal \Zbb \cap [a,b]$.

\section{Formal \covering step and convergence result of \cdsm}
\label{section:main_result_1}

This section formalizes the \cdsm as Algorithm~\ref{algo:cdsm} and its convergence analysis in Theorem~\ref{theorem:convergence_cdsm}.
Theorem~\ref{theorem:convergence_cdsm} is based on Property~\ref{property:DCP} provided by the \covering step, and on Assumption~\ref{assumption:problem} regarding the objective function~$f$.
The proof of Theorem~\ref{theorem:convergence_cdsm} follows in Section~\ref{section:proof}.

The \cdsm matches the usual \dsm in most of its aspects.
The only novelty lies in the \covering step, with its parameter~$r \in \Rbb_+^*$, detailed in Section~\ref{section:comments_covering_step}.
Our choices for the other steps and parameters follow simple instances of mesh-based \dsm~\cite[Part~3]{AuHa2017} and sufficient decrease-based \dsm~\cite[Section~7.7]{CoScVibook} respectively, but they may be designed as in many such \dsm.
We discuss the \cdsm in Remark~\ref{remark:cdsm}.

\smallskip

\begin{algorithm}[ht]
    \caption{\cdsm (\covering \dsm) solving Problem~\eqref{problem:P}.}
    \label{algo:cdsm}
    \begin{algorithmic}
    \State \textbf{Initialization}: \vspace{0.1cm}
    \State \hfill \begin{minipage}{0.96\linewidth}
        set a covering radius~$r \in \Rbb_+^*$; set the trial points history~$\Hcal^0 \defequal \emptyset$; \\
        set the incumbent solution and poll radius~$(x^0,\delta^0) \in X \times \Rbb_+^*$; set~$\underline{\delta}^0 \defequal \delta^0$; \\
        set~$\tau \in {]}0,1{[} \cap \Qbb$, and set~$\Mcal: \Rbb_+ \to 2^{\Rbb^n}$ and~$\rho: \Rbb_+ \to \Rbb_+$ according to one of
        \begin{itemize}[topsep=0pt,partopsep=0pt,noitemsep]
            \item the \textit{mesh-based \dsm}:~$\Mcal(\nu) \defequal \min\{\nu,\frac{\nu^2}{\delta^0}\}\Zbb^n$ and~$\rho(\nu) \defequal 0$ for all~$\nu \in \Rbb_+$;
            \item the \textit{sufficient decrease-based \dsm}:~$\Mcal(\nu) \defequal \Rbb^n$ and~$\rho(\nu) \defequal \min\{\nu,\frac{\nu^2}{\delta^0}\}$ for all~$\nu \in \Rbb_+$.
        \end{itemize}
    \end{minipage} \vspace{0.2cm}
    \For{$k \in \Nbb$}:
    \State
        \covering \textbf{step}:
        \State \hfill
        \begin{minipage}{0.93\linewidth}
            set~$\DcalC^k \subseteq \Mcal(\underline{\delta}^k) \cap \cl(\Bcal_r)$ nonempty and finite;
            set~$\TcalC^k \defequal \pointplusset{x^k}{\DcalC^k}$; set~$\tC^k \in \argmin f(\TcalC^k)$; \\
            if~$f(\tC^k) < f(x^k)-\rho(\underline{\delta}^k)$, then set~$t^k \defequal \tC^k$ and~$\TcalS^k = \TcalP^k \defequal \emptyset$ and skip to the \update step;
        \end{minipage}
    \vspace{0.1cm}
    \State
        \search \textbf{step}:
        \State \hfill
        \begin{minipage}{0.93\linewidth}
            set~$\DcalS^k \subseteq \Mcal(\underline{\delta}^k)$ empty or finite;
            if~$\TcalS^k \defequal \pointplusset{x^k}{\DcalS^k}$ is nonempty, then set~$\tS^k \in \argmin f(\TcalS^k)$; \\
            if also~$f(\tS^k) < f(x^k)-\rho(\underline{\delta}^k)$, then set~$t^k \defequal \tS^k$ and~$\TcalP^k \defequal \emptyset$ and skip to the \update step;
        \end{minipage}
    \vspace{0.1cm}
    \State
        \poll \textbf{step}:
        \State \hfill
        \begin{minipage}{0.93\linewidth}
            set~$\DcalP^k \subseteq \Mcal(\underline{\delta}^k) \cap \cl(\Bcal_{\delta^k})$ a positive basis of~$\Rbb^n$;
            set~$\TcalP^k \defequal \pointplusset{x^k}{\DcalP^k}$;
            set~$\tP^k \in \argmin f(\TcalP^k)$; \\
            if~$f(\tP^k) < f(x^k)-\rho(\underline{\delta}^k)$, then set~$t^k \defequal \tP^k$, otherwise set~$t^k \defequal x^k$;
        \end{minipage}
    \vspace{0.1cm}
    \State \update \textbf{step}:
        \State \hfill
        \begin{minipage}{0.93\linewidth}
            set~$\Hcal^{k+1} \defequal \Hcal^k \cup \Tcal^k$, where~$\Tcal^k \defequal \TcalC^k \cup \TcalS^k \cup \TcalP^k$;
            set~$x^{k+1} \defequal t^k$; \\
            set~$\delta^{k+1} \defequal \frac{1}{\tau}\delta^k$ if~$x^k \neq t^k$ and~$\delta^{k+1} \defequal \tau\delta^k$ otherwise;
            set~$\underline{\delta}^{k+1} \defequal \min\{\underline{\delta}^k,\delta^{k+1}\}$.
        \end{minipage}
    \EndFor
    \end{algorithmic}
\end{algorithm}

Algorithm~\ref{algo:cdsm} generates a sequence~$(x^k,\delta^k,\Hcal^k)_{k \in \Nbb}$ from which we may define
\begin{equation*}
    \Rcal \defequal \left\{\underset{k \in K^*}{\lim} x^k : K^* \subseteq \Nbb ~\mbox{indexes a refining subsequence of}~ (x^k,\delta^k)_{k \in \Nbb}\right\}
    \qquad \mbox{and} \qquad
    \Hcal \defequal \underset{k \in \Nbb}{\cup} \Hcal^k,
\end{equation*}
and a carefully constructed \covering step ensures that these two sets satisfy the next Property~\ref{property:DCP}.
A practical construction scheme to indeed meet this property is given in Section~\ref{section:comments_covering_step/construction_schemes}.

\begin{property}[Dense covering provided by Algorithm~\ref{algo:cdsm} with a well designed \covering step]
    \label{property:DCP}
    The inclusion
    \begin{equation*}
        \Bcal_r(x^*) \subseteq \cl(\Hcal)
    \end{equation*}
    holds for all~$x^* \in \Rcal$, where~$r \in \Rbb_+^*$ is the \covering radius chosen in the \textbf{initialization} step.
\end{property}

Theorem~\ref{theorem:convergence_cdsm} also requires the following assumptions about~$f$.
Assumptions~\refbis{assumption:problem}{f_bound} and~\refbis{assumption:problem}{f_low} match usual assumptions used in the literature, while the unusual Assumption~\refbis{assumption:problem}{X} is discussed in Section~\ref{section:comments_assumption_partition}.

\begin{assumption}[on the objective function~$f$]
    \label{assumption:problem}
    The objective function~$f$ in Problem~\eqref{problem:P} is such that
    \begin{enumerate}[label=\alph*)]
        \item\label{assumption:problem:f_bound}
            $f$ is bounded below and has bounded sublevel sets;
        \item\label{assumption:problem:f_low}
            the restriction $f_{|X} : X \to \Rbb$ is lower semicontinuous;
        \item\label{assumption:problem:X}
            $X$ admits a partition~$X = \partitioncup_{i=1}^{N} X_i$ (where~$N \in \Nbb^* \cup \{+\infty\}$) such that, for all~$i \in \ll1,N\rr$,~$X_i$ is an \textit{ample continuity set of~$f$} (that is,~$X_i$ is ample and the restriction~$f_{|X_i} : X_i \to \Rbb$ is continuous).
    \end{enumerate}
\end{assumption}

\begin{theorem}
    \label{theorem:convergence_cdsm}
    Under Assumption~\ref{assumption:problem}, Algorithm~\ref{algo:cdsm} generates at least one refining subsequence and, if Property~\ref{property:DCP} holds, then, for all~$K^* \subseteq \Nbb$ indexing a refining subsequence, the corresponding refined point~$x^*$ satisfies
    \begin{equation*}
        \underset{k \in K^*}{\lim} f(x^k) \ = \ 
        \left\{\begin{array}{rl}
            \min f(\Bcal_r(x^*))         {~= f(x^*)} & \mbox{if}~ x^* \in X, \\
            \inf f(\Bcal_r(x^*)) \phantom{~= f(x^*)} & \mbox{if}~ x^* \notin X.
        \end{array}\right.
    \end{equation*}
\end{theorem}

We conclude this section with two remarks about the \cdsm and Theorem~\ref{theorem:convergence_cdsm} respectively.

\begin{remark}
    \label{remark:cdsm}
    First, for ease of presentation, Algorithm~\ref{algo:cdsm} allows little freedom in the designs of the mesh~$\Mcal$, the sufficient decrease function~$\rho$ and the rule to update the poll radius~$\delta^k$.
    Nevertheless, their designs may follow any mesh-based \dsm~\cite[Part~3]{AuHa2017} and any sufficient decrease-based \dsm~\cite[Section~7.7]{CoScVibook}.
    A variant of Algorithm~\ref{algo:cdsm} with the generic definition of all these elements appears as Algorithm~\ref{algo:cdsm_generic} in Appendix~\ref{appendix:materials/algo}.
    Theorem~\ref{theorem:convergence_cdsm} remains valid when Algorithm~\ref{algo:cdsm} is replaced by Algorithm~\ref{algo:cdsm_generic}.
    Second, the mesh is defined in the space of directions in Algorithm~\ref{algo:cdsm}, while in the literature it is usually defined in the space of variables directly, but centered at the incumbent solution.
    This variation simplifies our presentation and is transparent.
    Third, the \covering step fits in the framework of the \search step.
    Consequently, the \cdsm inherits all the properties of the usual \dsm.
    Last, we stress that the \cdsm relies on the smallest poll radius~$\underline{\delta}^k$ in addition to the current poll radius~$\delta^k$.
    This dependency in~$(\underline{\delta}^k)_{k \in \Nbb}$ is mandatory in the proof of Theorem~\ref{theorem:convergence_cdsm} (precisely, in Proposition~\ref{proposition:limit_f_inf_neighborhood}).
\end{remark}

\begin{remark}
    \label{remark:theorem}
    When Assumption~\ref{assumption:problem} holds, Property~\ref{property:DCP} ensures that all refined points are local solutions to Problem~\eqref{problem:P}, in the usual sense for those lying in~$X$ and in a generalized sense for the others.
    In practice, most \dsm usually generate exactly one refined point which is moreover a local solution to Problem~\eqref{problem:P}.
    Nevertheless, to our best knowledge, no \dsm from the literature ensures Property~\ref{property:DCP} (and thus satisfies Theorem~\ref{theorem:convergence_cdsm}).
    Then, at the time of writing, the only \dsm that is guaranteed to generate a local solution under Assumption~\ref{assumption:problem} is the \cdsm relying on the \covering step defined in Section~\ref{section:comments_covering_step}.
\end{remark}

\section{Proof of Theorem~\ref{theorem:convergence_cdsm}}
\label{section:proof}

\noindent{\bf Preliminary results.}
The \covering step is a specific \search step, thus Algorithm~\ref{algo:cdsm} inherits all the properties of the usual \dsm.
Hence, the next Proposition~\ref{proposition:refining_sequence} holds.
It is stated as~\cite[Theorem~8.1]{AuHa2017} for the mesh-based \dsm and as~\cite[Corollary~7.2]{CoScVibook} for the sufficient decrease-based \dsm.

\begin{proposition}
    \label{proposition:refining_sequence}
    Under Assumption~\refbis{assumption:problem}{f_bound}, Algorithm~\ref{algo:cdsm} generates at least one refining subsequence.
\end{proposition}

The rest of this section relies on Proposition~\ref{proposition:limit_f_inf_neighborhood}, which shows what Property~\ref{property:DCP} provides and settles the ground for the proof of Theorem~\ref{theorem:convergence_cdsm}.
The proof of Proposition~\ref{proposition:limit_f_inf_neighborhood} involves an auxiliary topological claim that is proved in Proposition~\ref{proposition:ample_sets} in Appendix~\ref{appendix:propositions}.

\begin{proposition}\label{proposition:limit_f_inf_neighborhood}
    Under Assumption~\ref{assumption:problem}, if Algorithm~\ref{algo:cdsm} satisfies Property~\ref{property:DCP}, then, for all refined points~$x^*$, it holds that~$\lim_{k \in \Nbb} f(x^k) \leq f(x)$ for all~$x \in \Bcal_r(x^*)$.
\end{proposition}
\begin{proof}
    First,~$f^* \defequal \lim_{k \in \Nbb} f(x^k) \in \Rbb$ exists since~$(f(x^k))_{k \in \Nbb}$ decreases by construction and is bounded below by Assumption~\refbis{assumption:problem}{f_bound}.
    Second, let~$x^*$ be a refined point generated by Algorithm~\ref{algo:cdsm} satisfying Property~\ref{property:DCP} and take~$x \in \Bcal_r(x^*)$.
    The result holds if~$x \notin X$, so assume that~$x \in X_i$ for some~$i \in \ll1,N\rr$.
    Let~$(y_\ell)_{\ell \in \Nbb}$ converging to~$x$ with~$y_\ell \in \Bcal_r(x^*) \cap X_i  \cap \Hcal \setminus \{x\}$ for all~$\ell \in \Nbb$ (it exists by Proposition~\refbis{proposition:ample_sets}{dense_in_ample_minus_point} applied with~$\Scal_1 \defequal \Bcal_r(x^*)$,~$\Scal_2 \defequal X_i$ and~$\Scal_3 \defequal \Hcal$).
    Let~$\kappa(\ell) \defequal \min \{k \in \Nbb : y_\ell \in \Tcal^k\}$ for all~$\ell \in \Nbb$.
    Then~$(\kappa(\ell))_{\ell \in \Nbb}$ diverges to~$+\infty$, since~$\Tcal^k$ is finite for all~$k \in \Nbb$ and every subsequence of~$(y_\ell)_{\ell \in \Nbb}$ takes infinitely many values (since~$(y_\ell)_{\ell \in \Nbb}$ converges to~$x$ with~$y_\ell \neq x$ for all~$\ell \in \Nbb$).
    Moreover,~$(\underline{\delta}^{\kappa(\ell)})_{\ell \in \Nbb}$ converges to~$0$ since~$\liminf_{k \in \Nbb} \delta^k = 0$ as a consequence of Proposition~\ref{proposition:refining_sequence}.
    In addition, for all~$\ell \in \Nbb$ we have~$f(t^{\kappa(\ell)}) \geq f(x^{\kappa(\ell)}) - \rho(\underline{\delta}^{\kappa(\ell)})$ if iteration~$\kappa(\ell)$ fails and~$f(t^{\kappa(\ell)}) = f(x^{\kappa(\ell)+1})$ otherwise, and~$f(y_\ell) \geq f(t^{\kappa(\ell)})$ by construction.
    Hence, for all~$\ell \in \Nbb$ we have~$f(y_\ell) \geq f(x^{\kappa(\ell)+1}) -  \rho(\underline{\delta}^{\kappa(\ell)})$.
    The result follows by taking~$\ell \to +\infty$, since~$(f(y_\ell))_{\ell \in \Nbb}$ converges to~$f(x)$ by continuity of~$f_{|X_i}$ and~$(f(x^{\kappa(\ell)+1}))_{\ell \in \Nbb}$ converges to~$f^*$ and~$(\rho(\underline{\delta}^{\kappa(\ell)}))_{\ell \in \Nbb}$ converges to~$0$.
\end{proof}

\noindent{\bf Proof of Theorem~\ref{theorem:convergence_cdsm}.}
Consider that Assumption~\ref{assumption:problem} is satisfied.
Proposition~\ref{proposition:refining_sequence} states that at least one refining subsequence is generated.
Assume that Property~\ref{property:DCP} holds.
Let~$K^* \subseteq \Nbb$ indexing a refining subsequence,~$x^*$ denoting its refined point, and let~$f^* \defequal \lim_{k \in K^*} f(x^k) = \lim_{k \in \Nbb} f(x^k)$.
Let us show that~$f^* = \inf f(\Bcal_r(x^*))$.
We have~$f^* \geq \inf f(\Bcal_r(x^*))$ since~$x^k \in \Bcal_r(x^*)$ for all~$k \in K^*$ large enough, and~$f^* \leq \inf f(\Bcal_r(x^*))$ is proved by contradiction: if~$f^* > \inf f(\Bcal_r(x^*))$, then there exists~$x^\sharp \in \Bcal_r(x^*)$ such that~$f^* > f(x^\sharp)$, but then~$f^* > f(x^\sharp) \geq f^*$ by Proposition~\ref{proposition:limit_f_inf_neighborhood}.
Note that this already concludes the proof of the case where~$x^* \notin X$.
Now assume that~$x^* \in X$.
Then~$f^* \geq f(x^*)$ by Assumption~\refbis{assumption:problem}{f_low} and~$f^* \leq f(x^*)$ by Proposition~\ref{proposition:limit_f_inf_neighborhood}, so~$f(x^*) = f^* = \inf f(\Bcal_r(x^*)) = \min f(\Bcal_r(x^*))$.
\hfill \qed

\section{Discussion on the \covering step and Property~\ref{property:DCP}}
\label{section:comments_covering_step}

This section discusses the \covering step.
Section~\ref{section:comments_covering_step/invalid_revealing} highlights the differences between the \covering step and the \revealing step it is inspired from.
Section~\ref{section:comments_covering_step/sufficient_condition_schemes} provides a sufficient condition about a construction scheme for the \covering step to ensure Property~\ref{property:DCP}.
Section~\ref{section:comments_covering_step/construction_schemes} provides a tractable construction scheme for the \covering step which checks this sufficient condition.

\subsection{The \revealing step provides Property~\ref{property:DCP} a posteriori in some cases only}
\label{section:comments_covering_step/invalid_revealing}

As stated in Section~\ref{section:intro}, the \covering step is inspired from the \revealing step in a \dsm that we hereafter call the \revealing \dsm (\rdsm)~\cite[Algorithm~1]{AuBoBo22}.
The \rdsm satisfies Property~\ref{property:DCP} when exactly one refined point is generated~\cite[Lemma~2]{AuBoBo22}.
Nevertheless, we show in this section that the \rdsm may fail to ensure Property~\ref{property:DCP} when more than one refined point is generated.

Expressed in our notation, the \revealing step in~\cite{AuBoBo22} relies on a sequence~$(\Dcal_\ell)_{\ell \in \Nbb}$ with~$\Dcal_\ell \subset \cl(\Bcal_r)$ finite for all~$\ell \in \Nbb$ and such that~$\Bcal_r \subseteq \cl(\cup_{\ell \in \Nbb} \Dcal_\ell)$.
For all~$k \in \Nbb$, it defines~$\DcalC^k \defequal \round(\Dcal_{\ell(k)}, \Mcal(\underline{\delta}^k))$ as the rounding of~$\Dcal_{\ell(k)}$ onto~$\Mcal(\underline{\delta}^k)$, where~$\ell(k)$ denotes the number of iterations indices~$u < k$ such that~$\underline{\delta}^{u+1} < \underline{\delta}^u$.
Thus, the \revealing step ensures that
\begin{equation*}
    \Bcal_r \subseteq \cl\left(\underset{k \in \Nbb}{\cup} \DcalC^k\right)
\end{equation*}
when the \rdsm generates at least one refining subsequence, which is certified a priori by Proposition~\ref{proposition:refining_sequence}.
Nevertheless, this property states only the dense intersection of the trial directions history with~$\Bcal_r$.
This may not lead to Property~\ref{property:DCP} when the \rdsm generates more than one refined point.
The following Example~\ref{example:revealing_multiple_refined}, illustrated in Figure~\ref{figure:revealing_fails}, confirms this observation.

\begin{example}
    \label{example:revealing_multiple_refined}
    Consider the objective function $f: \Rbb^2 \to \Rbb$ defined by~$f(x) \defequal \textnorm{x}_{\infty}$ for all~$x \in \Rbb^2$ and the algorithmic parameters
        $r \defequal 1$,
        $x^0 \defequal -3\1$ where~$\1 \defequal (1,1)$, 
        $\delta^0 \defequal 1$, 
        $\Mcal(\nu) \defequal \min\{\nu,\nu^2\}\Zbb^2$ and~$\rho(\nu) \defequal 0$ for all~$\nu \in \Rbb_+^*$, 
        and~$\tau \defequal \frac{1}{2}$.
    For simplicity, we do not increase the \poll radius in case of a successful iteration (that is, we force~$\delta^{k+1} \defequal \delta^k$ if~$x^k \neq t^k$).
    Define the \poll step as~$\DcalP^k \defequal \{(\pm\delta^k,0),(0,\pm\delta^k)\}$ for all~$k \in \Nbb$.
    Define the \revealing step such that the sequence~$(\Dcal_\ell)_{\ell \in \Nbb}$ satisfies
    \begin{equation*}
        \forall q \in \Nbb, \qquad
        \Dcal_{4q}   \subset \Rbb_+\times\Rbb_+, \quad
        \Dcal_{4q+1} \subset \Rbb_-\times\Rbb_+, \quad
        \Dcal_{4q+2} \subset \Rbb_-\times\Rbb_-, \quad
        \Dcal_{4q+3} \subset \Rbb_+\times\Rbb_-.
    \end{equation*}
    Define the \search step at each iteration~$k \in \Nbb$ so that~$\TcalS^k \defequal \emptyset$ if~$k \notin 3\Nbb$, and so that~$\TcalS^{3q} \defequal \{\tS^{3q}\}$ for all~$k = 3q \in 3\Nbb$, where
    \begin{equation*}
        \tS^{3q} \defequal \left(-1\right)^q \left(1+2^{-q}\right)\1.
    \end{equation*}
    
    This instance of \rdsm has a predictable behavior.
    Using induction, one can show that 
    \begin{equation*}
        \forall q \in \Nbb, \quad
        \left\{\begin{array}{llll}
                x^{3q\phantom{+1}}      = (-1)^{q-1}\left(1+2^{-(q-1)}\right)\1, &
                \delta^{3q\phantom{+1}} = \phantom{\frac{1}{2}}4^{-q}, &
                \DcalC^{3q\phantom{+1}} = \Dcal_{2q}, &
                \mbox{\search success,}
            \\
                x^{3q+1}                = (-1)^{q\phantom{-1}}\left(1+2^{-q}\right)\1, &
                \delta^{3q+1}           = \phantom{\frac{1}{2}}4^{-q}, &
                \DcalC^{3q+1}           = \Dcal_{2q}, &
                \mbox{iteration fails,}
            \\
                x^{3q+2}                = (-1)^{q\phantom{-1}}\left(1+2^{-q}\right)\1, &
                \delta^{3q+2}           = \frac{1}{2}4^{-q}, &
                \DcalC^{3q+2}           = \Dcal_{2q+1}, &
                \mbox{iteration fails.}
        \end{array}\right.
    \end{equation*}
    Thus, there are two refined points~$x^*_+  = \1$ and~$x^*_- = -\1$, and none is a local minimizer of~$f$.
    However, this does not contradict Theorem~\ref{theorem:convergence_cdsm} because Property~\ref{property:DCP} is not satisfied.
    Indeed, it holds that
    \begin{equation*}
        \cl(\Hcal) ~\cap~ {]} -1 , 1 {[}~ ^2 = \emptyset
        \quad \mbox{while} \quad
        \forall r \in \Rbb_+^*,~
        \Bcal_r(x^*_+) = {]}  1-r ,  1+r {[}~ ^2
        ~\mbox{and}~
        \Bcal_r(x^*_-) = {]} -1-r , -1+r {[}~ ^2.
    \end{equation*}
\end{example}

\begin{figure}[!hb]
    \centering
    \includegraphics[width=0.5\linewidth]{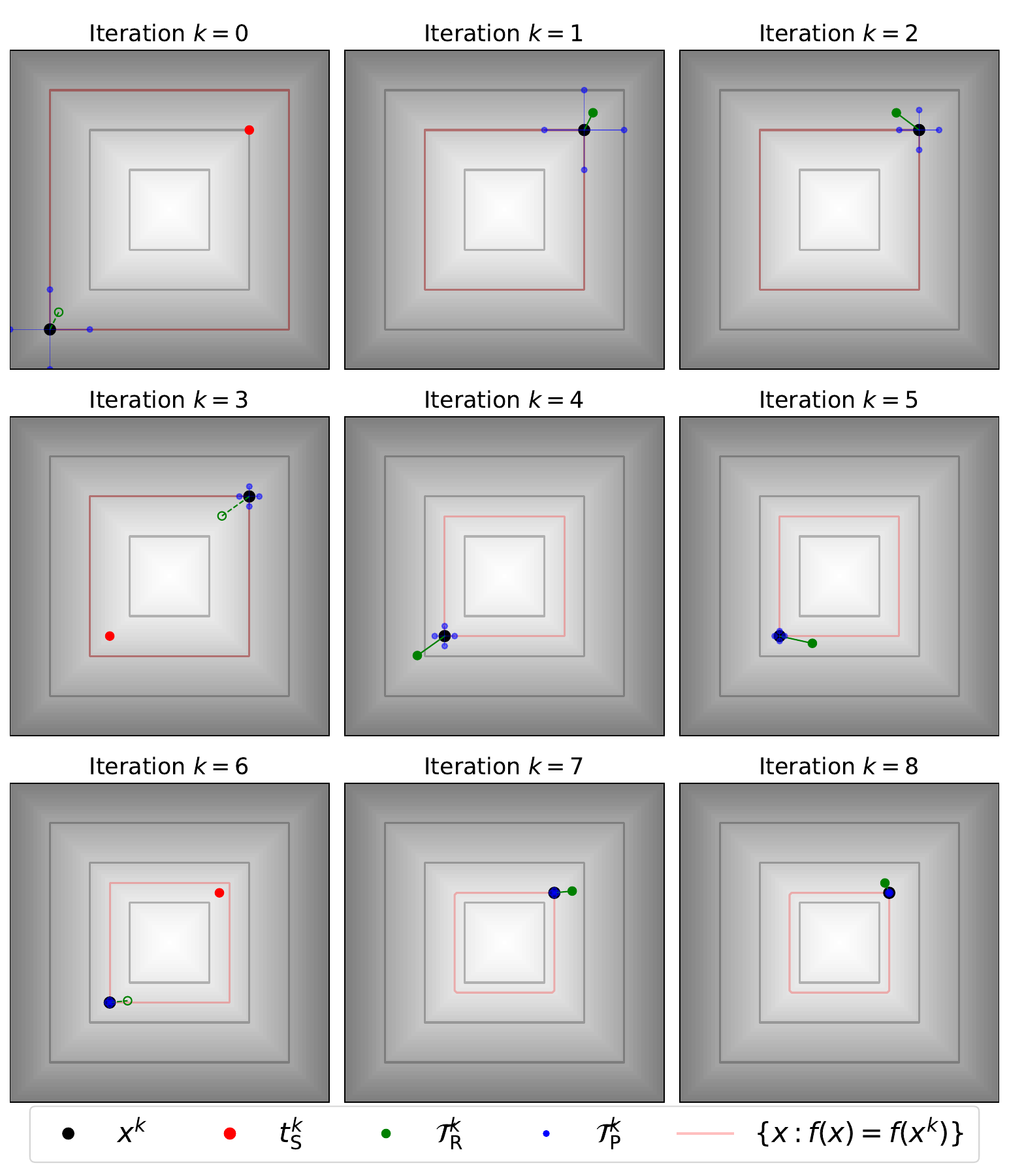}
    \caption{First eight iterations of the \rdsm in the context of Example~\ref{example:revealing_multiple_refined}. The set~$\TcalR^k$ of \revealing trial points is dotted at each iteration~$k \in 3\Nbb$ to highlight that these points are actually not evaluated by the \rdsm. Indeed, the \search step trial point~$\tS^k$ is evaluated before and makes the iteration a success.}
    \label{figure:revealing_fails}
\end{figure}

Example~\ref{example:revealing_multiple_refined} shows that the \revealing step from~\cite{AuBoBo22} focuses only on the asymptotic density of the trial directions, and that this may not translate to the density of the trial points in some neighborhoods of the refined points.
The original instance of the \revealing step provided in~\cite{AuBaKo22} differs from those in~\cite{AuBoBo22} (in~\cite{AuBaKo22}, the \revealing step draws at each iteration a direction in~$\cl(\Bcal_r)$ according to the independent uniform distribution), but it also fails to provide Property~\ref{property:DCP} when more than one refined point exists.
Accordingly, in Section~\ref{section:comments_covering_step/sufficient_condition_schemes}, we study schemes for the \covering step that ensure instead the dense intersection of the trial points with a neighborhood of all refined points.

\subsection{Sufficient condition to ensure a priori that the \cdsm satisfies Property~\ref{property:DCP}}
\label{section:comments_covering_step/sufficient_condition_schemes}

This section focuses on \covering step instances relying only, at each iteration~$k \in \Nbb$, on the current couple~$(x^k,\underline{\delta}^k)$ and the current history~$\Hcal^k \defequal \cup_{\ell < k} \Tcal^\ell$ (where~$\Tcal^\ell$ denotes the set of all trial points evaluated at iteration~$\ell \in \Nbb$).
Proposition~\ref{proposition:sufficient_condition_covering} proves that, if the \covering step relies on a \covering oracle from Definition~\ref{definition:oracle}, illustrated in Figure~\ref{figure:covering_oracle}, then all executions of the \cdsm satisfy Property~\ref{property:DCP}.

\begin{definition}[\covering oracle]
    \label{definition:oracle}
    For all~$r \in \Rbb_+^*$, a function~$\Obb: \Rbb^n \times 2^{\Rbb^n} \to 2^{\Rbb^n}$ is said to be a \textit{\covering oracle} of radius~$r$ if there exists~$\beta \in {]}0,1]$ such that, for all points~$x \in \Rbb^n$ and all sets~$\emptyset \neq \Scal \subseteq \Rbb^n$,
    \begin{equation*}
        \Obb(x,\Scal) ~\mbox{is compact}
        ~\mbox{and}~
        \emptyset \neq \Obb(x,\Scal) \subseteq \cl(\Bcal_r)
        \qquad \mbox{and} \qquad
        \dfrac
            {\max_{d \in \Obb(x,\Scal)}~ \dist\left(x+d, \Scal\right)}
            {\max_{d \in \cl(\Bcal_r)} ~ \dist\left(x+d, \Scal\right)}
        \ \geq \ \beta.
    \end{equation*}
\end{definition}

\begin{figure}[!hb]
    \centering
    \includegraphics[width=0.70\linewidth, trim=05 05 05 05, clip]{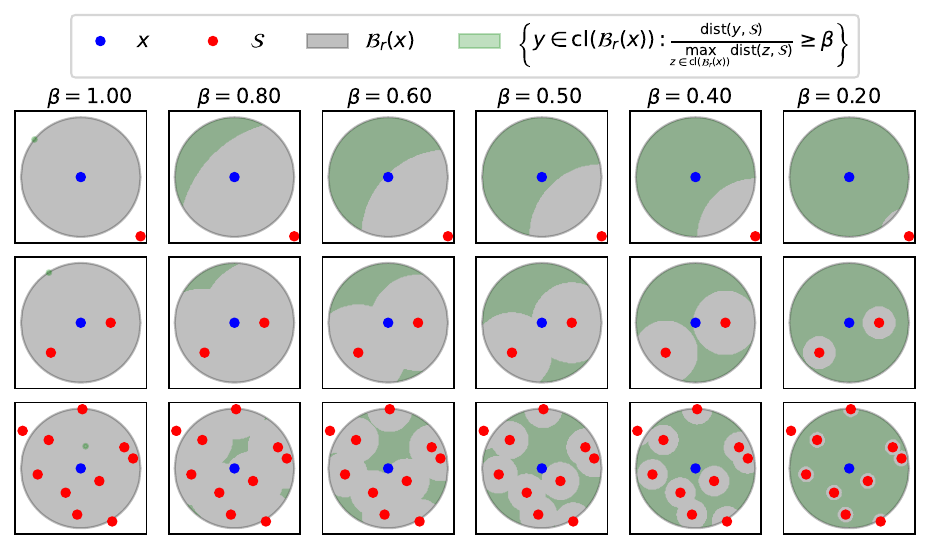}
    \caption{Illustration of Definition~\ref{definition:oracle}.
    A function~$\Obb$ satisfies Definition~\ref{definition:oracle} at~$(x,\Scal)$ for the value~$\beta$ if~$\Obb(x,\Scal)$ contains at least one direction~$d$ such that~$x+d$ lies in the set represented in green.}
    \label{figure:covering_oracle}
\end{figure}

\begin{proposition}
    \label{proposition:sufficient_condition_covering}
    For all functions~$\Obb$ satisfying Definition~\ref{definition:oracle}, if Algorithm~\ref{algo:cdsm} defines its \covering step as
    \begin{equation*}
        \forall k \in \Nbb, \qquad
        \DcalC^k \defequal \round\left( \Obb\left(x^k, \Hcal^k\right), \Mcal(\underline{\delta}^k) \right),
    \end{equation*}
    then it ensures Property~\ref{property:DCP}.
\end{proposition}
\begin{proof}
    Consider the framework of Proposition~\ref{proposition:sufficient_condition_covering}.
    Let~$x^* \in \Rcal$ and~$K^*\subseteq \Nbb$ indexing an associated refining subsequence, such that~$\max_{d \in \Rbb^n} \textnorm{d-\round(d,\Mcal(\underline{\delta}^k))} \leq r$ and~$x^k \in \cl(\Bcal_r(x^*))$ for all~$k \in K^*$.
    For all~$k \in \Nbb$, let~$d_\Obb^k \in \Obb(x^k,\Hcal^k)$ such that~$\dist(x^k+d_\Obb^k,\Hcal^k) = \max_{d \in \Obb(x^k,\Hcal^k)} \dist(x^k+d,\Hcal^k)$.
    Let also~$\dC^k \defequal \round(d_\Obb^k,\Mcal(\underline{\delta}^k)) \in \DcalC^k$ and~$\tC^k \defequal x^k+\dC^k \in \TcalC^k \subseteq \Tcal^k$.
    Observe that, when~$k \to +\infty$ in~$K^*$,
    \begin{equation*}
        \norm{x^*-x^k} \to 0, \qquad
        \norm{d_\Obb^k-\dC^k} \to 0
        \qquad \mbox{and} \qquad
        \min_{\ell < k} \norm{\tC^k-\tC^\ell} \to 0.
    \end{equation*}
    The first two claims result from~$K^*$ indexing a refining subsequence.
    The third claim is proved by contradiction.
    Assume that there exists~$K \subseteq K^*$ infinite and~$\varepsilon > 0$ such that~$\min_{\ell < k} \norm{\tC^k-\tC^\ell} \geq \varepsilon$ for all~$k \in K$.
    Then,~$\textnorm{\tC^k-\tC^\ell} \geq \varepsilon$ for all~$(k,\ell) \in K^2$ with~$\ell < k$.
    Nevertheless, for all~$k \in K$, it holds that~$\textnorm{\tC^k-x^*} = \textnorm{x^k+\dC^k+d_\Obb^k-d_\Obb^k-x^*} \leq \textnorm{x^k-x^*} + \textnorm{d_\Obb^k} + \textnorm{\dC^k-d_\Obb^k} \leq 3r$, so all elements of~$(\tC^k)_{k \in K}$ lie in the compact set~$\cl(\Bcal_{3r}(x^*))$.
    This situation contradicts the Bolzano-Weierstrass theorem applied to~$(\tC^k)_{k \in K}$.
    Then, to conclude the proof, remark that, for all~$k \in K^*$,
    \begin{eqnarray*}
        \max_{d \in \Bcal_r} \dist\left(x^*+d, \Hcal\right)
        & \leq & \max_{d \in \cl(\Bcal_r)} \dist\left(x^*+d, \Hcal^k\right) \\
        & \leq & \norm{x^*-x^k} + \max_{d \in \cl(\Bcal_r)} \dist\left(x^k+d, \Hcal^k\right) \\
        & \leq & \norm{x^*-x^k} + \frac{1}{\beta} \max_{d \in \Obb(x^k,\Hcal^k)} \dist\left(x^k+d,\Hcal^k\right) \\
        &   =  & \norm{x^*-x^k} + \frac{1}{\beta} \dist\left(x^k+d_\Obb^k,\Hcal^k\right) \\
        & \leq & \norm{x^*-x^k} + \frac{1}{\beta} \left( \norm{d_\Obb^k-\dC^k} + \dist\left(x^k+\dC^k,\Hcal^k\right) \right) \\
        &   =  & \norm{x^*-x^k} + \frac{1}{\beta}\norm{d_\Obb^k-\dC^k} + \frac{1}{\beta}\dist\left(\tC^k, \underset{\ell < k}{\cup} \Tcal^\ell\right) \\
        & \leq & \norm{x^*-x^k} + \frac{1}{\beta}\norm{d_\Obb^k-\dC^k} + \frac{1}{\beta}\min_{\ell < k} \norm{\tC^k-\tC^\ell}.
    \end{eqnarray*}
    Then, taking the limit as~$k \to +\infty$ in~$K^*$ leads to the desired result that~$\Bcal_r(x^*) \subseteq \cl(\Hcal)$.
\end{proof}

Proposition~\ref{proposition:sufficient_condition_covering} depends only on the oracle~$\Obb$ driving the \covering step, so it may be checked a priori (prior to the execution of the \cdsm).
In contrast, Property~\ref{property:DCP} must be checked a posteriori, as this requires the whole sequence~$(x^k,\delta^k,\Hcal^k)_{k \in \Nbb}$ to compute~$\Rcal$ and~$\Hcal$ and check Property~\ref{property:DCP}.
Thus, Proposition~\ref{proposition:sufficient_condition_covering} allows for a simple way to ensure a priori that Property~\ref{property:DCP} will be satisfied.
The construction scheme for the \covering step we introduce in Section~\ref{section:comments_covering_step/construction_schemes} relies on Definition~\ref{definition:oracle} and Proposition~\ref{proposition:sufficient_condition_covering}.

\subsection{Construction of a \covering step instance usable in practice}
\label{section:comments_covering_step/construction_schemes}

In practice, an instance of the \covering step must satisfy two criteria.
For theoretical consistency, the resulting trial points history~$\Hcal$ must satisfy Property~\ref{property:DCP}, and for practical efficiency, the number of \covering trial points must be small at each iteration.
This section proposes a scheme that meets these two criteria.
We follow the guideline from Section~\ref{section:comments_covering_step/sufficient_condition_schemes}, that is,~$\DcalC^k \defequal \round(\Obb\left(x^k, \Hcal^k\right), \Mcal(\underline{\delta}^k))$ at each~$k \in \Nbb$, where~$\Obb$ is a \covering oracle from Definition~\ref{definition:oracle}.
Our scheme designs a tractable expression for~$\Obb$, valid for all \covering radii~$r \in \Rbb_+^*$.

The baseline scheme we suggest for the \covering step relies on the following oracle:
\begin{equation}
    \label{oracle:scheme_covering}
    \forall x \in \Rbb^n, \quad
    \forall \Scal \subseteq \Rbb^n, \qquad
    \Obb(x,\Scal) \defequal \underset{d \in \cl(\Bcal_r)}{\argmax}~ \dist\left(x+d, \Scal\right).
\end{equation}
Oracle~\ref{oracle:scheme_covering} is a \covering oracle, as it satisfies Definition~\ref{definition:oracle} with~$\beta \defequal 1$.
Then, Proposition~\ref{proposition:sufficient_condition_covering} claims that a \covering step instance relying on Oracle~\ref{oracle:scheme_covering} ensures Property~\ref{property:DCP}.
At each~$k \in \Nbb$, this instance selects~$\DcalC^k$ as the set of all~$\dC^k \in \cl(\Bcal_r)$ such that~$\tC^k \defequal x^k+\dC^k$ is the farthest possible from the set~$\Hcal^k$ of all past trial points (or, in the case of the mesh-based \cdsm,~$\DcalC^k$ is the rounding of all these directions~$\dC^k$ onto the mesh).
Moreover, Property~\ref{property:DCP} remains valid if we actually compute only one such direction.
The computation of~$\Obb$ is costly when~$\Scal$ contains many elements, but it may be alleviated.

Computing Oracle~\ref{oracle:scheme_covering} is a continuous and piecewise affine problem, for all~$(x,\Scal) \in \Rbb^n \times 2^{\Rbb^n}$.
This problem admits smooth surrogates, such as~$\overline{\Obb}(x,\Scal) \defequal \argmax_{d \in \cl(\Bcal_r)} \Sigma_{s \in \Scal} \frac{-1}{\norm{x+d-s}}$.
Moreover, the computation of~$\Obb(x^k,\Hcal^k)$ at each iteration~$k \in \Nbb^*$ may start from the point~$\tC^{k-1}$ calculated at the preceding iteration.
Also, in the case of the mesh-based \cdsm, the set~$\Obb(x,\Scal)$ is rounded onto the mesh, so it may be more relevant to restrict the search to direction lying on the mesh directly.
For all~$(x,\Scal) \in \Rbb^n \times 2^{\Rbb^n}$ and all mesh parameters~$\nu \in \Rbb_+^*$, we may seek for the directions~$d \in \Mcal(\nu) \cap \cl(\Bcal_r)$ maximizing~$\dist(x+d,\Scal)$.
This problem is combinatorial because of the discrete nature of the mesh, but it may be solved using a \textit{distance transform algorithm} such as~\cite{Meijster2000}.
This algorithm works in a number of operations linear with the cardinality of~$\Mcal(\nu) \cap \cl(\Bcal_r)$, and most are feasible in parallel.

Oracle~\ref{oracle:scheme_covering} may be adapted to ease its computation.
Let~$\alpha \in {]}0,1]$ and define
\begin{equation}
    \label{oracle:scheme_covering_weakened}
    \Obb_\alpha(x,\Scal) \defequal \left\{
        d_\alpha \in \cl(\Bcal_r) :
        \dfrac
            {\dist\left(x+d_\alpha, \Scal\right)}
            {\underset{d \in \cl(\Bcal_r)}{\max} \dist\left(x+d, \Scal\right)}
        \geq \alpha
    \right\},
\end{equation}
for all~$(x,\Scal) \in \Rbb^n \times 2^{\Rbb^n}$ with~$\Scal \neq \emptyset$.
Oracle~\ref{oracle:scheme_covering_weakened} satisfies Definition~\ref{definition:oracle} with~$\beta \defequal \alpha$, then it is a \covering oracle.
The set~$\Obb_\alpha(x,\Scal)$ usually contains infinitely many elements, but recall that in practice we do not need to compute more than one.
A simple heuristic approach to localize such an element~$d_\alpha$ may use a grid search on a grid thin enough, or some \textit{space-filling sequences} such as the Halton sequence~\cite{Ha60}.
It is presumably easier to compute an element of the set defined by Oracle~\ref{oracle:scheme_covering_weakened} than one of the set defined by Oracle~\ref{oracle:scheme_covering}, especially when~$\alpha$ is chosen close to~$0$ and the points in~$\Scal$ are well spread into~$\cl(\Bcal_r(x))$.

Oracles~\ref{oracle:scheme_covering} and~\ref{oracle:scheme_covering_weakened} may be altered for some marginal gain, by setting~$\delta r \in \Rbb_+^*$ small and considering
\begin{equation}
    \label{oracle:scheme_covering_truncated}
    \Obb^{\delta r}(x,\Scal) \defequal \Obb\left(x, \Scal \cap \cl(\Bcal_{r+\delta r}(x))\right)
    \qquad \mbox{and} \qquad
    \Obb^{\delta r}_\alpha(x,\Scal) \defequal \Obb_\alpha\left(x, \Scal \cap \cl(\Bcal_{r+\delta r}(x))\right)
\end{equation}
respectively, for all~$x \in \Rbb^n$ and all~$\Scal \subseteq \Rbb^n$ intersecting~$\cl(\Bcal_{r+\delta r}(x))$.
These alterations reduce the number of elements to consider in the computation of the point-set distance, and they remain \covering oracles (Definition~\ref{definition:oracle} holds for~$\Obb^{\delta r}$ with~$\beta \defequal \frac{\delta r}{2r+\delta r}$, and for~$\Obb^{\delta r}_\alpha$ with~$\beta \defequal \alpha\frac{\delta r}{2r+\delta r}$).

As desired, the four schemes above ensure a priori that any \cdsm relying on them satisfies Property~\ref{property:DCP} a posteriori, and that the additional cost per iteration induced by the \covering step is small.
Indeed, our schemes construct a \covering step instance evaluating~$1$ point per iteration, while for comparison the \poll step considers at least~$n+1$ points per iteration since it relies on a positive basis of~$\Rbb^n$.
Moreover, in a blackbox context, the cost to compute Oracle~\ref{oracle:scheme_covering} or any of its alterations may be negligible anyway, since the bottleneck in this context is the cost to evaluate~$f(x)$ for any~$x \in X$ while the computation of Oracle~\ref{oracle:scheme_covering} involves no call to the objective function.

Let us stress that, in practice, we may perform a \revealing step such as in~\cite{AuBaKo22,AuBoBo22} instead of a \covering step relying on a \covering oracle.
We refer to Section~\ref{section:comments_covering_step/invalid_revealing} for details about the \revealing step.
Indeed, the \revealing step ensures Property~\ref{property:DCP} when the algorithm generates exactly one refined point~\cite[Lemma~2]{AuBoBo22}, which is the usual behavior observed in practice.
Moreover, the computational cost required by the \revealing step is almost null.
Nevertheless, despite its more expensive cost, Oracle~\ref{oracle:scheme_covering} ensures that the trial points are well spread in a neighborhood of the current incumbent solution at each iteration.
This contrasts with the \revealing step, which offers no such guarantee.

Let us discuss also the \covering radius~$r$.
All~$r \in \Rbb_+^*$ are accepted, but fine-tuning this value is a problem-dependent concern.
Two default choices when no information about Problem~\eqref{problem:P} is available are~$r \defequal \frac{\delta^0}{10}$ or~$r \defequal \delta^0$.
An extreme case occurs when~$r$ is larger than the diameter of the sublevel set of~$f$ of level~$f(x^0)$ (Assumption~\refbis{assumption:problem}{f_bound} ensures that this diameter is finite).
In this case, Theorem~\ref{theorem:convergence_cdsm} ensures that the \cdsm returns a global solution to Problem~\eqref{problem:P}.
However, an overly large value of~$r$ is impracticable.
The larger~$r$ is, the more likely the \cdsm eventually escapes poor local solutions, while in contrast, the smaller~$r$ is, the faster the \covering step covers~$\Bcal_r(x^*)$ well.
In order to ensure that all points in~$\Bcal_r(x^*)$ are at distance at most~$0 <\varepsilon \leq r$ of a \covering trial point, the \covering step must generate at least~$(\frac{r}{\varepsilon})^n$ trial points in~$\Bcal_r(x^*)$ (and space them evenly).
For a fixed value of~$\varepsilon$, this number rapidly grows as~$r$ increases.
Then, the value of~$r$ impacts the reliability of the \cdsm as follows: the \cdsm asymptotically covers any ball of radius~$r$ it visits infinitely often with smaller balls of any radius~$0 < \varepsilon \leq r$, but to do so, each of these balls must be visited during at least~$(\frac{r}{\varepsilon})^n$ iterations.
One may also consider a sequence~$(r^k)_{k \in \Nbb}$ instead of a fixed~$r$, provided that~$0 < \underline{r} \defequal \inf_{k \in \Nbb} r^k < +\infty$.
In that case, one must replace~$\Bcal_r(x^*)$ by~$\Bcal_{\underline{r}}(x^*)$ in Property~\ref{property:DCP} and Theorem~\ref{theorem:convergence_cdsm}.

\section{Formal \covering step suited to many classes of \dfo methods}
\label{section:main_result_2}

Section~\ref{section:main_result_1} states Theorem~\ref{theorem:convergence_cdsm} for the \dsm only.
Nevertheless, it is possible to preserve a similar theorem by considering many other \dfo algorithms differing from \dsm, provided that they are enhanced with a \covering step.
More precisely, a convergence result similar to Theorem~\ref{theorem:convergence_cdsm} holds when the \cdsm is replaced by any algorithm fitting in the broad algorithmic framework given as Algorithm~\ref{algo:algo_with_covering}.

\smallskip

\begin{algorithm}[ht]
    \caption{Generic \dfo algorithm with a \covering step to solve Problem~\eqref{problem:P}.}
    \label{algo:algo_with_covering}
    \begin{algorithmic}
    \State \textbf{Initialization}: \vspace{0.1cm}
    \State \hfill \begin{minipage}{0.96\linewidth}
        set a covering radius~$r \in \Rbb_+^*$ and a \covering oracle~$\Obb$ (that is,~$\Obb$ satisfying Definition~\ref{definition:oracle}); \\
        set the incumbent solution~$x^0 \in X$ and the trial points history~$\Hcal^0 \defequal \emptyset$.
    \end{minipage} \vspace{0.2cm}
    \For{$k \in \Nbb$}:
    \State
        \covering \textbf{step}:
        \State \hfill
        \begin{minipage}{0.93\linewidth}
            set~$\DcalC^k \defequal \Obb(x^k,\Hcal^k)$;
            set~$\TcalC^k \defequal \pointplusset{x^k}{\DcalC^k}$ and~$\tC^k \in \argmin f(\TcalC^k)$; \\
            if~$f(\tC^k) < f(x^k)$, then set~$t^k \defequal \tC^k$ and~$\TcalO^k \defequal \emptyset$ and skip to the \update step;
        \end{minipage}
    \vspace{0.1cm}
    \State
        \algoname{optional} \textbf{step}:
        \State \hfill
        \begin{minipage}{0.93\linewidth}
            set~$\DcalO^k \subseteq \Rbb^n$ empty or finite;
            if~$\TcalO^k \defequal \pointplusset{x^k}{\DcalO^k}$ is nonempty, then set~$\tO^k \in \argmin f(\TcalO^k)$; \\
            if~$f(\tO^k) < f(x^k)$, then set~$t^k \defequal \tO^k$, otherwise set~$t^k \defequal x^k$;
        \end{minipage}
    \vspace{0.1cm}
    \State \update \textbf{step}:
        \State \hfill
        \begin{minipage}{0.93\linewidth}
            set~$x^{k+1} \defequal t^k$;
            set~$\Hcal^{k+1} \defequal \Hcal^k \cup \TcalC^k \cup \TcalO^k$.
        \end{minipage}
    \EndFor
    \end{algorithmic}
\end{algorithm}

Algorithm~\ref{algo:algo_with_covering} generates a sequence~$(x^k,\Hcal^k)_{k \in \Nbb}$ from which we may define
\begin{equation*}
    \Acal \defequal \left\{\underset{k \in K^*}{\lim} x^k : K^* \subseteq \Nbb ~\mbox{indexes a converging subsequence of}~ (x^k)_{k \in \Nbb}\right\}
    \qquad \mbox{and} \qquad
    \Hcal \defequal \underset{k \in \Nbb}{\cup} \Hcal^k,
\end{equation*}
and an equivalent of Property~\ref{property:DCP} holds, since the \covering step follows a \covering oracle.
More precisely, it holds that~$\Bcal_r(x^*) \subset \cl(\Hcal)$ for all~$x^* \in \Acal$.

Unlike the \cdsm, Algorithm~\ref{algo:algo_with_covering} does not project the oracle trial points onto a mesh and does not discriminate the trial points via a sufficient decrease function.
Moreover, it involves no sequence of poll radii and no mandatory \poll step, so it cannot be studied via the notion of refining subsequence.
Nevertheless, its convergence analysis, formalized in Theorem~\ref{theorem:convergence_algo_with_covering}, is similar to those of the \cdsm.
Also, Theorem~\ref{theorem:convergence_algo_with_covering} is valid for the \cdsm, provided that its \covering step is defined exactly as in Algorithm~\ref{algo:algo_with_covering}.

\begin{theorem}
    \label{theorem:convergence_algo_with_covering}
    Under Assumption~\ref{assumption:problem}, all~$x^* \in \acc \neq \emptyset$ generated by Algorithm~\ref{algo:algo_with_covering} and all~$K^* \subseteq \Nbb$ such that~$(x^k)_{k \in K^*}$ converges to~$x^*$ satisfy
    \begin{equation*}
        \underset{k \in K^*}{\lim} f(x^k) \ = \ 
        \left\{\begin{array}{rl}
            \min f(\Bcal_r(x^*))         {~= f(x^*)} & \mbox{if}~ x^* \in X, \\
            \inf f(\Bcal_r(x^*)) \phantom{~= f(x^*)} & \mbox{if}~ x^* \notin X.
        \end{array}\right.
    \end{equation*}
\end{theorem}
\begin{proof}
    Consider Algorithm~\ref{algo:algo_with_covering} under Assumption~\ref{assumption:problem}.
    First, Algorithm~\ref{algo:algo_with_covering} generates~$\acc \neq \emptyset$ since all elements of~$(x^k)_{k \in \Nbb}$ lie in the closure of the sublevel set of~$f$ of level~$f(x^0)$, which is compact because of Assumption~\refbis{assumption:problem}{f_bound}.
    Second, we match the notation of Section~\ref{section:proof} by denoting by~$\Mcal(\nu) \defequal \Rbb^n$ and~$\rho(\nu) \defequal 0$ for all~$\nu \in \Rbb_+$, and we set~$\underline{\delta}^k \defequal 0$ for all~$k \in \Nbb$.
    This proof consists in reading former propositions (and their proofs) to observe that they remain valid by replacing a refining subsequence by a converging one in this setting.
    The proof of Proposition~\ref{proposition:sufficient_condition_covering} requires a set~$K^* \subseteq \Nbb$ indexing a refining sequence only to ensure that~$(\underline{\delta}^k)_{k \in K^*}$ converges to~$0$, so it remains valid in our current setting.
    Then, Proposition~\ref{proposition:sufficient_condition_covering} claims that the trial points history~$\Hcal$ satisfies~$\Bcal_r(x^*) \subseteq \cl(\Hcal)$ for all~$x^* \in \Acal$.
    Similarly, the proofs of Proposition~\ref{proposition:limit_f_inf_neighborhood} and of Theorem~\ref{theorem:convergence_cdsm} remain valid for our current setting, since again a set~$K^* \subseteq \Nbb$ indexing a refining sequence is required in these proofs only to ensure that~$(\underline{\delta}^k)_{k \in K^*}$ converges to~$0$.
    Hence, Theorem~\ref{theorem:convergence_algo_with_covering} is proved by reading the proof of Theorem~\ref{theorem:convergence_cdsm} in the current setting.
\end{proof}

In practice, Algorithm~\ref{algo:algo_with_covering} allows skipping the \covering step and allows a perturbation of the \covering oracle (for example, by a projection onto a mesh).
However, the \covering step must be performed infinitely often, in the sense that for all~$x^* \in \Acal$, there must exists~$K^* \subseteq \Nbb$ such that~$(x^k)_{k \in K^*}$ converges to~$x^*$ and the \covering step is performed at all iterations~$k \in K^*$.
Moreover, the \covering oracle must not be excessively perturbed by, for example, rounding infinitely often onto a coarse mesh.
Precisely, if Algorithm~\ref{algo:algo_with_covering} relies on a positive sequence~$(\nu^k)_{k \in \Nbb}$ and sets~$\DcalC^k \defequal \round(\Obb(x^k,\Hcal^k), \Mcal(\nu^k))$ or accepts~$x^{k+1} \defequal \tC^k$ only if~$f(\tC^k) < f(x^k) - \rho(\nu^k)$ at each iteration~$k \in \Nbb$, then it is mandatory that~$(\nu^k)_{k \in \Nbb}$ decreases to~$0$.
This echoes our comment in Remark~\ref{remark:cdsm} that the \cdsm relies on the sequence of smallest poll radii~$(\underline{\delta}^k)_{k \in \Nbb}$ in addition to those of the current poll radii~$(\delta^k)_{k \in \Nbb}$.
Indeed, the mesh-based \cdsm rounds the \covering oracle into the mesh parameterized by~$(\underline{\delta}^k)_{k \in \Nbb}$ and the sufficient decrease-based \cdsm discriminates trial points by a sufficient decrease depending on~$(\underline{\delta}^k)_{k \in \Nbb}$, and (unlike the sequence~$(\delta^k)_{k \in \Nbb}$) the sequence~$(\underline{\delta}^k)_{k \in \Nbb}$ is guaranteed to converge to~$0$ under Assumption~\ref{assumption:problem}.

We conclude this section by stressing that Theorem~\ref{theorem:convergence_algo_with_covering} is applicable in wide range of situations.
Indeed, Section~\ref{section:comments_assumption_partition/tight} below shows that Assumption~\ref{assumption:problem} is quite weak, and Section~\ref{section:comments_covering_step/construction_schemes} provides an explicit scheme for the \covering oracle compatible with many \dfo algorithms.
Moreover, Theorem~\ref{theorem:convergence_algo_with_covering} shows that the \covering step is, roughly speaking, a self-sufficient algorithmic step that ensures the local optimality of all accumulation points returned by any algorithm it is fitted into.

\section{Discussion on Assumption~\refbis{assumption:problem}{X}}
\label{section:comments_assumption_partition}

This section discusses our novel assumption describing the continuity sets of~$f$, that is, Assumption~\refbis{assumption:problem}{X}.
In Section~\ref{section:comments_assumption_partition/comparison_prior_work}, we prove that Assumption~\refbis{assumption:problem}{X} is strictly weaker than similar assumptions considered in former work~\cite{AuBaKo22,AuBoBo22,ViCu2012}.
In Section~\ref{section:comments_assumption_partition/tight}, we show that Assumption~\refbis{assumption:problem}{X} is tight.

\subsection{Comparison of Assumption~\refbis{assumption:problem}{X} with similar assumptions in prior work}
\label{section:comments_assumption_partition/comparison_prior_work}

In this section, we compare Assumption~\refbis{assumption:problem}{X} to similar assumptions considered by former work~\cite{AuBaKo22,AuBoBo22,ViCu2012}.
Precisely, we show that Assumption~\refbis{assumption:problem}{X} is strictly weaker than either~\cite[Assumption~4.4]{AuBaKo22} and~\cite[Assumption~1]{AuBoBo22}.
The work~\cite{ViCu2012} is not considered since~\cite{AuBoBo22} is an extension of it.

First, let us compare Assumption~\refbis{assumption:problem}{X} to~\cite[Assumption~4.4]{AuBaKo22}, recalled below as Assumption~\ref{assumption:AuBaKo}.
We prove in Proposition~\ref{proposition:assumption_current_weaker_AuBaKo} that Assumption~\refbis{assumption:problem}{X} is strictly weaker than Assumption~\ref{assumption:AuBaKo}.

\begin{assumption}[Assumption 4.4 in~\cite{AuBaKo22}]
    \label{assumption:AuBaKo}
    There exists~$N \in \Nbb^* \cup \{+\infty\}$~nonintersecting open sets~$X_i$ such that~$\cl(X) = \cup_{i=1}^{N} \cl(X_i)$ and~$f_{|X_i}$ is continuous for all~$i \in \ll1,N\rr$ and, for all~$x \in X$, there exists~$j \in \ll1,N\rr$ such that~$x \in \cl(X_j)$ and~$f_{|X_j \cup \{x\}}$ is continuous.
\end{assumption}

\begin{proposition}
    \label{proposition:assumption_current_weaker_AuBaKo}
    Assumption~\refbis{assumption:problem}{X} is strictly weaker than Assumption~\ref{assumption:AuBaKo}.
\end{proposition}
\begin{proof}
    Suppose that Assumption~\ref{assumption:AuBaKo} holds.
    Denote by~$(X_i)_{i=1}^{N}$ the family it provides.
    For all~$i \in \ll1,N\rr$, let~$\cl_f(X_i) \defequal \{x \in \cl(X_i) : f_{|X_i \cup \{x\}}$ is continuous$\}$.
    Let~$I(x) \defequal \min \{i \in \ll1,N\rr : x \in \cl_f(X_i)\}$ for all~$x \in X$.
    Then let~$Y_i \defequal \{x \in X : I(x) = i\}$ for every~$i \in \ll1,N\rr$.
    Thus~$(Y_i)_{i=1}^{N}$ satisfies all the requirements in Assumption~\refbis{assumption:problem}{X} (see Proposition~\ref{proposition:AuBaKo_implies_our_technical}), so Assumption~\refbis{assumption:problem}{X} is weaker than Assumption~\ref{assumption:AuBaKo}.
    Now, to prove that Assumption~\ref{assumption:AuBaKo} is not weaker than Assumption~\refbis{assumption:problem}{X}, consider the case
    \begin{equation*}
        \fct{f}
            {x}
            {X \defequal [-1,1] \setminus \{0\}}
            {\dfrac{1}{i} ~\mbox{if}~ \textabs{x} \in \left]\dfrac{1}{i+1},\dfrac{1}{i}\right] ~\mbox{for some}~ i \in \Nbb^*.}
            {\Rbb}
    \end{equation*}
    The continuity sets of~$f$ are~$X_i \defequal [\frac{-1}{i},\frac{-1}{i+1}{[} \cup {]}\frac{1}{i+1},\frac{1}{i}]$ for all~$i \in \Nbb^*$.
    Assumption~\refbis{assumption:problem}{X} holds since~$X_i$ is ample for all~$i \in \Nbb^*$.
    Nevertheless Assumption~\ref{assumption:AuBaKo} does not hold since the continuity sets must be adapted as~$Y_i \defequal \int(X_i)$ for all~$i \in \Nbb^*$ to be open, but then~$\cl(X) = [-1,1] \neq \cup_{i=1}^{\infty} \cl(Y_i) = [-1,1] \setminus \{0\}$.
\end{proof}

Second, let us compare Assumption~\refbis{assumption:problem}{X} to~\cite[Assumption~1]{AuBoBo22}, reformulated below in Assumption~\ref{assumption:AuBoBo}.
We prove in Proposition~\ref{proposition:assumption_current_weaker_AuBoBo} that Assumption~\refbis{assumption:problem}{X} is strictly weaker than Assumption~\ref{assumption:AuBoBo}.

\begin{assumption}[Global reformulation of Assumption 1 in~\cite{AuBoBo22}]
    \label{assumption:AuBoBo}
    The set~$X$ admits a partition~$X = \partitioncup_{i=1}^{N} X_i$ (where~$N \in \Nbb^* \cup \{+\infty\}$) such that, for all~$i \in \ll1,N\rr$,~$X_i$ is a \textit{continuity set of~$f$ with the interior cone property} (that is,~$X_i$ satisfies Definition~\ref{definition:ICP} below and~$f_{|X_i} : X_i \to \Rbb$ is continuous).
\end{assumption}

\begin{definition}[Interior cone property and exterior cone property]
    \label{definition:ICP}
    A set~$\Scal \subseteq \Rbb^n$ is said to have the \textit{interior cone property} (ICP) if
    \begin{equation*}
        \forall x \in \partial\Scal, \quad
        \exists \left\{\begin{array}{ll}
            \multicolumn{2}{l}{\Ucal \subseteq \Sbb^n ~\mbox{nonempty, open in the topology induced by}~ \Sbb^n} \\
            \Kcal \defequal \Rbb_+^*\Ucal ~\mbox{the cone generated by}~ \Ucal, & \Kcal_x \defequal \pointplusset{x}{\Kcal} \\
            \Ocal ~\mbox{an open neighborhood of}~ 0 ~\mbox{in}~ \Rbb^n,        & \Ocal_x \defequal \pointplusset{x}{\Ocal}
            \end{array}\right.
        : \qquad
        \left(\Kcal_x \cap \Ocal_x\right) \subseteq \Scal.
    \end{equation*}
    Similarly,~$\Scal$ is said to have the \textit{exterior cone property} (ECP) if~$(\Rbb^n \setminus \Scal)$ has the ICP.
\end{definition}

Assumption~\ref{assumption:AuBoBo} differs from~\cite[Assumption~1]{AuBoBo22} in three aspects.
Let us stress those and argue that they do not spoil the assumption.
First, Assumption~\ref{assumption:AuBoBo} is a global statement, while~\cite[Assumption~1]{AuBoBo22} states a local property but for each~$x \in X$.
Second,~\cite[Assumption~1]{AuBoBo22} requires Lipschitz-continuity of~$f$ on each continuity set, while Assumption~\ref{assumption:AuBoBo} calls for continuity only.
Indeed~\cite{AuBoBo22} requires Lipschitz-continuity only to evaluate some generalized derivatives.
We do not consider those in the current work, so we weaken the assumption accordingly.
Third,~\cite[Assumption~1]{AuBoBo22} requires an \textit{exterior} cone property for~$X \setminus X_i$ for all~$i \in \ll1,N\rr$, while Assumption~\ref{assumption:AuBoBo} demands an \textit{interior} cone property for~$X_i$.
Moreover the exterior cone property required in~\cite[Assumption~1]{AuBoBo22} is~\cite[Definition~4.1]{ViCu2012}, which differs from Definition~\ref{definition:ICP}.
However, these two approaches are equivalent (see Proposition~\ref{proposition:new_ECP_equiv_old}).
Hence, Assumption~\ref{assumption:AuBoBo} and~\cite[Assumption~1]{AuBoBo22} differ only in the nature of the continuity of~$f$ on each continuity set.

\begin{proposition}
    \label{proposition:assumption_current_weaker_AuBoBo}
    Assumption~\refbis{assumption:problem}{X} is strictly weaker than Assumption~\ref{assumption:AuBoBo}.
\end{proposition}
\begin{proof}
    For all~$\Scal \subseteq \Rbb^n$, if~$\Scal$ has the ICP, then~$\Scal$ is ample, but the reciprocal implication may fail (see Proposition~\ref{proposition:AuBoBo_implies_our_technical}).
    The result follows directly.
\end{proof}

\subsection{Tightness of Assumption~\refbis{assumption:problem}{X}}
\label{section:comments_assumption_partition/tight}

In this section, we show that Assumption~\refbis{assumption:problem}{X} is tight, in the sense that the conclusion of Theorem~\ref{theorem:convergence_algo_with_covering} may not hold if Assumption~\refbis{assumption:problem}{X} is not satisfied.

\begin{proposition}
    \label{proposition:assumption_current_tight}
    If Assumption~\refbis{assumption:problem}{X} does not hold, then the conclusion of Theorem~\ref{theorem:convergence_algo_with_covering} may not hold.
\end{proposition}
\begin{proof}
    Let us develop a counterexample.
    Let~$n = 2$ and~$f: \Rbb^2 \to \Rbb$ be defined by
    \begin{equation*}
        \forall x \in X \defequal \Rbb^2, \qquad
        f(x) \defequal \left\{\begin{array}{ll}
            \textabs{x_1-1}+\textabs{x_2}-1, & \mbox{if}~ x \in X_1 \defequal (\Rbb_-\times\Rbb) \cup (\Rbb^*_+\times\{0\}), \\
            \textabs{x_1}+\textabs{x_2},     & \mbox{if}~ x \in X_2 \defequal X \setminus X_1.
        \end{array}\right.
    \end{equation*}
    Assumptions~\refbis{assumption:problem}{f_bound} and~\refbis{assumption:problem}{f_low} hold, but Assumption~\refbis{assumption:problem}{X} does not since~$X_1$ is locally thin and thus not ample.
    Consider an instance of Algorithm~\ref{algo:algo_with_covering} such that~$(\DcalC^k \cup \DcalO^k) \cap (\Rbb^*_+\times\{0\}) = \emptyset$ for all~$k\in \Nbb$ and starting from the origin.
    This instance remains at the origin, since it evaluates only points in~$\cl(\int(X_1)) \cup X_2 = \Rbb^2 \setminus (\Rbb_+^* \times \{0\})$ and the origin is the global minimizer of the restriction of~$f$ to~$\cl(\int(X_1)) \cup X_2$.
    In that situation, the origin is a refined point lying in~$X$ but is not a local minimizer of~$f$, which contradicts the conclusion of Theorem~\ref{theorem:convergence_algo_with_covering}.
\end{proof}

Assumption~\refbis{assumption:problem}{X} allows a broad class of discontinuous functions, as it only rejects discontinuous functions for which at least one of the continuity sets is locally thin.
A relaxation as a local assumption holding only at the refined points is possible, but such relaxation can be checked only a posteriori since it is impossible to determine the refined points a priori.
Hence our framework cannot be broadened using only information available a priori.
Nevertheless, we leave the associated adaptation of Theorem~\ref{theorem:convergence_algo_with_covering} as Assertion~\ref{assertion:convergence_cdsm_no_assumption_X} below.
Its proof is similar to those provided in Section~\ref{section:main_result_2}.

\begin{assertion}
    \label{assertion:convergence_cdsm_no_assumption_X}
    Under Assumptions~\refbis{assumption:problem}{f_bound} and~\refbis{assumption:problem}{f_low}, all~$x^* \in \Acal \neq \emptyset$ generated by Algorithm~\ref{algo:algo_with_covering} and all sets~$K^* \subseteq \Nbb$ such that~$(x^k)_{k \in K^*}$ converges to~$x^*$ satisfy
    \begin{enumerate}[label=\alph*)]
        \item if~$\Xcal_i$ is ample for all~$i \in \ll1,N\rr$, then~$\lim_{k \in K^*} f(x^k) = \inf f(\Bcal_r(x^*))$;
        \item if moreover~$x^* \in \Xcal_i$ for some~$i \in \ll1,N\rr$, then~$\lim_{k \in K^*} f(x^k) = \min f(\Bcal_r(x^*)) = f(x^*)$;
    \end{enumerate}
    where we denote by~$(\Xcal_i)_{i=0}^{N}$ a partition of~$\Bcal_r(x^*)$, for some~$N \in \Nbb^* \cup \{+\infty\}$, such that~$\Xcal_0 \defequal \Bcal_r(x^*) \setminus X$ and~$\Xcal_i$ is a continuity set of~$f$ for all~$i \in \ll1,N\rr$.
\end{assertion}

\section{General comments and main extensions for future work}
\label{section:general_comments}

Theorem~\ref{theorem:convergence_cdsm} is stronger than most results about the \dsm in three aspects.
First, Theorem~\ref{theorem:convergence_cdsm} ensures the local optimality of all refined points while the literature usually checks necessary optimality conditions.
Second, Theorem~\ref{theorem:convergence_cdsm} involves Assumption~\ref{assumption:problem} on the problem only, while the \dsm~\cite[Theorems~3.1 and~3.2]{ViCu2012} and the \rdsm~\cite[Theorem~1]{AuBoBo22} require also an assumption on the execution of the algorithm (respectively, that~$\lim_{k \in K^*} f(x^k) = f(x^*)$ for all~$K^* \subseteq \Nbb$ indexing a refining subsequence with refined point~$x^*$, and that a unique refined point is generated).
Third, Theorem~\ref{theorem:convergence_cdsm} holds for all refining subsequences.
In contrast, the literature usually considers only refining subsequences such that the set of associated refined directions is dense in~$\Sbb^n$, and no way to identify such a refining subsequence is provided.

The assumption that~$f$ is lower semicontinuous may be relaxed.
The proof of Theorem~\ref{theorem:convergence_cdsm} highlights that, if only Assumption~\refbis{assumption:problem}{f_low} fails, then~$\lim_{k \in K^*} f(x^k) = \inf f(\Bcal_r(x^*))$ holds true for all~$x^* \in \Rcal$ and all~$K^* \subseteq \Nbb$ indexing an associated refining subsequence.
Assuming the lower semicontinuity of~$f$ at~$x^*$ recovers~$\inf f(\Bcal_r(x^*)) = \min f(\Bcal_r(x^*)) = f(x^*)$.
A similar claim holds true for Theorem~\ref{theorem:convergence_algo_with_covering}.

Only the \covering step matters to establish the convergence towards local solutions, in the sense formalized in Section~\ref{section:main_result_2}.
Many methods fit in Algorithm~\ref{algo:algo_with_covering} when enhanced with a \covering step.
Future work may implement a \covering step into Bayesian methods~\cite{ZhZi21GlobalBayesian}, model-based methods~\cite{AuHa20}, and most methods listed in~\cite{LaMeWi2019}.
The use of a few \covering step directions per iteration also shares similarities with \dfo line search algorithms~\cite{BeCaSc21} and reduced space algorithms~\cite{RoRo23ReducedSpaces}.
By coupling the \covering step with others ideas from the literature, several technical requirements in the \cdsm may be alleviated.
For example, the limitation that~$x^0 \in X$ is relaxable when Problem~\eqref{problem:P} is the \textit{extreme barrier} formulation of a constrained problem with quantifiable constraints~\cite{LedWild2015}, using for example the two-phase algorithm~\cite[Algorithm~12.1]{AuHa2017} or the \textit{progressive barrier}~\cite{AuDe09a}.

Although they are optional in Algorithm~\ref{algo:algo_with_covering}, the \search and \poll steps are important in practice.
The first allows for global exploration, and the second usually contributes to many successful iterations.
In contrast, the \covering step aims to ensure the asymptotic Property~\ref{property:DCP}, so a poor instance may be inefficient in finite time.
We conducted a numerical experiment on problems where the objective functions are alterations of~$f \defequal \textnorm{\cdot}$.
It appears that the addition of the \covering step makes negligible difference in performance when the objective function is not too erratic near the local solution the \cdsm approaches, but that the \covering step provides a gain of quality on the returned solution when the objective function has very thin continuity sets.
These experiments are reported to Appendix~\ref{appendix:materials/numerical}.
Future work may investigate these observations and quantify the relevance of each step.
Moreover, we only tested our baseline schemes for the \covering step schemes from Section~\ref{section:comments_covering_step/construction_schemes}.
Further investigations and careful implementations may provide more efficient schemes.
We may also study how well the \cdsm performs when Assumption~\refbis{assumption:problem}{X} holds but its stronger variant Assumption~\ref{assumption:AuBoBo} does not.
Presumably, the interior cone property from Assumption~\ref{assumption:AuBoBo} is important for practical efficiency.

Our schemes for the \covering oracle in Section~\ref{section:comments_covering_step/construction_schemes} ensure only that the trial points are well spread around the current incumbent solution at each iteration.
Another scheme, for example from an expected improvement~\cite{ZhXi20} or model-based techniques~\cite{AuHa20}, may also seek for points that are relevant candidates or that help to gather the shape of~$f$.
Another future work could alter Assumption~\refbis{assumption:problem}{X} so that a partition of~$X$ into continuity sets of~$f$ is explicitly provided.
A \covering oracle designed accordingly could therefore evaluate points in all continuity sets of~$f$, even if some are not ample.

The \covering step is also compatible with the \discomads algorithm~\cite{AuBaKo22}, which designs the original \revealing step for the purpose to detect discontinuities and repel its incumbent solution from those.
This discontinuities detection is more accurate with the \covering step than with the \revealing step, as the latter may fail to detect discontinuities when more than one refined point is eventually generated.
Also, when the \covering step relies on Oracle~\ref{oracle:scheme_covering}, it presumably has better practical guarantees to efficiently detect discontinuities than both the the original \revealing step from~\cite{AuBaKo22} and the adapted \revealing step from~\cite{AuBoBo22}.
Then, for reliability reasons, it may be safer to use the \discomads algorithm with a \covering step instead of a \revealing step.

We conclude this paper with the next Table~\ref{table:summary_cdsm_conceptual}.
It summarizes the conceptual differences between the usual \dsm and our \cdsm, and it highlights the aforementioned ideas to extend our work.

\medskip

\begin{table}[ht]
    \centering
    \setlength{\tabcolsep}{2pt} \renewcommand{\arraystretch}{1.5}
    \begin{tabular}{|| M{0.12\linewidth} || M{0.27\linewidth} | M{0.27\linewidth} | M{0.27\linewidth} ||}
        \hline \hline
        \multirow{2}{*}{step} & \multicolumn{3}{c||}{method} \\
        \cline{2-4}
        & \dsm & \cdsm & \dfo method with involved \covering step design \\ \hline \hline
        \search   &
            \multicolumn{3}{c||}{Optional. Allows the use of heuristics and globalization strategies.}
            \\ \hline
        \poll &
            Required. Converges towards refined points satisfying necessary optimality conditions.
            &
            \multicolumn{2}{c||}{\begin{minipage}{0.5\linewidth}\centering Optional. But in practice, it performs well in converging towards a good refined point. \end{minipage}}
            \\ \hline
        \covering &
            Undefined.
            &
            Required. Asymptotically ensures that all refined points are local solutions; low cost per iteration but lacks efficiency in practice.
            &
            Required. Asymptotically ensures that all refined points are local solutions; exploits information about the objective function.
            \\ \hline \hline
    \end{tabular}
    \caption{Differences between \dsm and \cdsm, and possible new methods with new \covering step goals.}
    \label{table:summary_cdsm_conceptual}
\end{table}

\clearpage\newpage
\appendix

\section{Supplementary materials}
\label{appendix:materials}

\subsection{Variant of Algorithm~\ref{algo:cdsm} with nonstringent definition of its parameters}
\label{appendix:materials/algo}

\medskip

\begin{algorithm}[ht]
    \caption{Generic \cdsm (\covering \dsm) solving Problem~\eqref{problem:P}.}
    \label{algo:cdsm_generic}
    \begin{algorithmic}
    \State \textbf{Initialization}: \vspace{0.1cm}
    \State \hfill \begin{minipage}{0.96\linewidth}
        set a covering radius~$r \in \Rbb_+^*$ and the trial points history as~$\Hcal^0 \defequal \emptyset$; \\
        set the incumbent solution and poll radius as~$(x^0,\delta^0) \in X \times \Rbb_+^*$, and set~$\underline{\delta}^0 \defequal \delta^0$; \\
        set~$\Mcal: \Rbb_+ \to 2^{\Rbb^n}$ and~$\rho: \Rbb_+ \to \Rbb_+$ and~$\Lambda \subset {]}0,1[$ and~$\Upsilon \subset [1,+\infty[$ according to one of
        \begin{itemize}[topsep=0pt,partopsep=0pt]
            \item the \textit{mesh-based \dsm}:~$\Mcal(\nu) \defequal \min\{\nu,\frac{\nu^2}{\delta^0}\}M\Zbb^p$ for all~$\nu \in \Rbb_+$, where~$p > n$ and~$M \in \Rbb^{n\times p}$ positively spans~$\Rbb^n$, and~$\rho(\cdot) \defequal 0$, and~$\Lambda \subseteq \{\tau^\ell : \ell \in \ll1,m\rr\}$ and~$\Upsilon \subseteq \{\tau^\ell : \ell \in \ll-m,0\rr\}$ where~$\tau \in \Qbb \cap {]}0,1[$ and~$m \in \Nbb^*$;
            \item the \textit{sufficient decrease-based \dsm}:~$\Mcal(\cdot) \defequal \Rbb^n$, and~$\rho$ increasing with~$0 < \rho(\nu) \in o(\nu)$ as~$\nu \searrow 0$, and~$\Lambda \subseteq [\underline{\lambda},\overline{\lambda}]$ and~$\Upsilon \subseteq [\underline{\upsilon},\overline{\upsilon}]$ where~$0 < \underline{\lambda} \leq \overline{\lambda} < 1 \leq \underline{\upsilon} \leq \overline{\upsilon} < +\infty$.
        \end{itemize}
    \end{minipage} \vspace{0.2cm}
    \For{$k \in \Nbb$}:
    \State
        \covering \textbf{step}:
        \State \hfill
        \begin{minipage}{0.93\linewidth}
            set~$\DcalC^k \subseteq \Mcal(\underline{\delta}^k) \cap \cl(\Bcal_r)$ nonempty and finite;
            set~$\TcalC^k \defequal \pointplusset{x^k}{\DcalC^k}$;
            set~$\tC^k \in \argmin f(\TcalC^k)$; \\
            if~$f(\tC^k) < f(x^k)-\rho(\underline{\delta}^k)$, then set~$t^k \defequal \tC^k$ and~$\TcalS^k = \TcalP^k \defequal \emptyset$ and skip to the \update step;
        \end{minipage}
    \vspace{0.1cm}
    \State
        \search \textbf{step}:
        \State \hfill
        \begin{minipage}{0.93\linewidth}
            set~$\DcalS^k \subseteq \Mcal(\underline{\delta}^k)$ empty or finite;
            if~$\TcalS^k \defequal \pointplusset{x^k}{\DcalS^k}$ is nonempty, then set~$\tS^k \in \argmin f(\TcalS^k)$; \\
            if also~$f(\tS^k) < f(x^k)-\rho(\underline{\delta}^k)$, then set~$t^k \defequal \tS^k$ and~$\TcalP^k \defequal \emptyset$ and skip to the \update step;
        \end{minipage}
    \vspace{0.1cm}
    \State
        \poll \textbf{step}:
        \State \hfill
        \begin{minipage}{0.93\linewidth}
            set~$\DcalP^k \subseteq \Mcal(\underline{\delta}^k) \cap \cl(\Bcal_{\delta^k})$ a positive basis of~$\Rbb^n$;
            set~$\TcalP^k \defequal \pointplusset{x^k}{\DcalP^k}$;
            set~$\tP^k \in \argmin f(\TcalP^k)$; \\
            if~$f(\tP^k) < f(x^k)-\rho(\underline{\delta}^k)$, then set~$t^k \defequal \tP^k$, otherwise set~$t^k \defequal x^k$;
        \end{minipage}
    \vspace{0.1cm}
    \State \update \textbf{step}:
        \State \hfill
        \begin{minipage}{0.93\linewidth}
            set~$\Hcal^{k+1} \defequal \Hcal^k \cup \Tcal^k$, where~$\Tcal^k \defequal \TcalC^k \cup \TcalS^k \cup \TcalP^k$;
            set~$x^{k+1} \defequal t^k$; \\
            set~$\delta^{k+1}$ as~$\delta^{k+1} \in \delta^k\Upsilon$ if~$t^k \neq x^k$ and~$\delta^{k+1} \in \delta^k\Lambda$ otherwise;
            set~$\underline{\delta}^{k+1} \defequal \min_{\ell \leq k} \delta^\ell$.
        \end{minipage}
    \EndFor
    \end{algorithmic}
\end{algorithm}

\smallskip

\subsection{Report of some numerical experiments regarding the \cdsm}
\label{appendix:materials/numerical}

This appendix compares the numerical performance of the \cdsm with the usual \dsm.
We consider two problems that illustrate typical situations a \dfo algorithm may face in a discontinuous context,
\begin{equation}
    \label{problem:P_numerical_1}
    \tag{$\Pbf^\mathbf{test}_1$}
    \problemoptimfree{\minimize}{x \in \Rbb^2}{
        \left\{\begin{array}{ll}
            \norm{x}_\infty+1 & \mbox{if}~ x_1 > 0, \\
            \norm{x}_\infty   & \mbox{otherwise,}
        \end{array}\right.
        }
\end{equation}
and
\begin{equation}
    \label{problem:P_numerical_2}
    \tag{$\Pbf^\mathbf{test}_2$}
    \problemoptimfree{\minimize}{x \in \Rbb^2}{
        \left\{\begin{array}{ll}
            \norm{x}_\infty+1 & \mbox{if}~ x_1 > 0, \\
            \norm{x}_\infty   & \mbox{if}~ x_1 \leq 0 ~\mbox{and}~ \norm{p(x)}_2 \leq \min\{\norm{q(x)}_2^2, \frac{1}{100}\}, \\
            +\infty           & \mbox{otherwise,}
        \end{array}\right.
        }
\end{equation}
where, for all~$x \in \Rbb^2$,~$p(x) \defequal \frac{\transpose{x}a}{\textnorm{a}^2}a$ denotes the projection of~$x$ onto the line directed by~$a \defequal (-1,1)$ and where~$q(x) \defequal x-p(x)$.
The objective functions of these two problems are represented on Figure~\ref{figure:illustration_problems}.

\begin{figure}[!ht]
    \centering
    \includegraphics[width=0.30\linewidth, trim=30 00 15 00, clip]{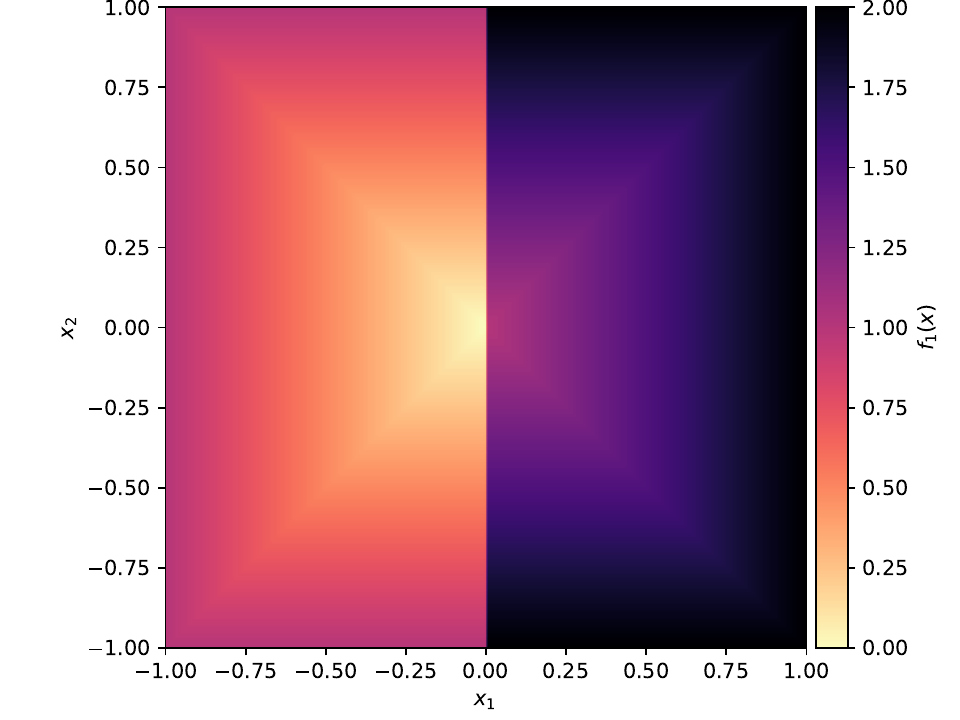}
    \hspace{1cm}
    \includegraphics[width=0.30\linewidth, trim=30 00 15 00, clip]{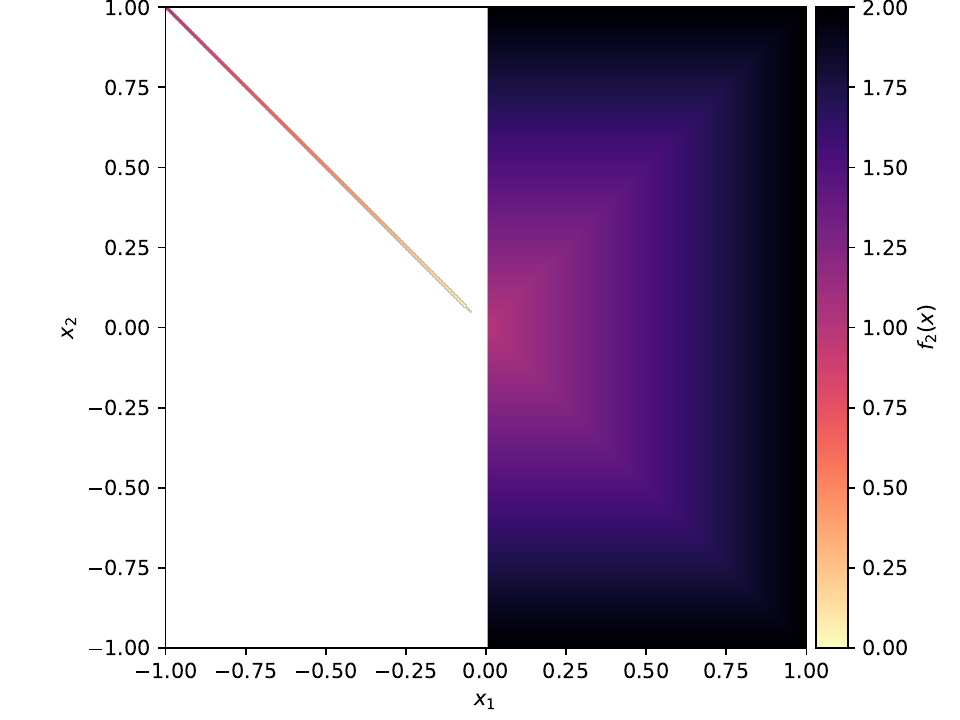}
    \hfill
    \caption{Objective functions of Problems~\eqref{problem:P_numerical_1} (on the left plot) and~\eqref{problem:P_numerical_2} (on the right plot).}
    \label{figure:illustration_problems}
\end{figure}

Our experiments study an instance of Algorithm~\ref{algo:cdsm}.
In agreement with the framework of Theorem~\ref{theorem:convergence_algo_with_covering}, we define~$\Mcal(\nu) \defequal \Rbb^n$ and~$\rho(\nu) \defequal 0$ for all~$\nu \in \Rbb_+$.
The incumbent is~$x^0 \defequal (98.7654321,12.3456789)$, the poll radius is~$\delta^0 \defequal 1$, the \covering radius is~$r \defequal \frac{\delta^0}{10}$, and the shrinking factor is~$\tau \defequal \frac{1}{2}$.
The \covering step considers, for each iteration~$k \in \Nbb$, all directions~$d$ on a grid over~$\cl(\Bcal_r)$ thin enough and selects one maximizing~$\dist(x^k+d,\Hcal^k)$, as in Section~\ref{section:comments_covering_step/construction_schemes}.
The \search step explores a momentum-based strategy with~$\DcalS^k \defequal \{3(x^k-x^{k-1})\}$ for all iteration~$k \in \Nbb^*$.
The \poll step follows the orthogonal polling from~\cite{AbAuDeLe09}.
The stopping criterion is that either~$\delta^k < 10^{-8}$ or~$k \geq 300$.
We compare~$10$ executions of this \cdsm with~$10$ executions of the associated \dsm skipping the \covering step.

For each problem and each execution of each algorithm, we record the smallest objective function value found after the current number of objective function evaluations, and the evolution of the ratio of respectively the number of \search successes, \covering successes, \poll successes and iterations failures over the current number of iterations achieved.

Figure~\ref{figure:numerical_2} shows the results for Problem~\eqref{problem:P_numerical_1}.
The discontinuity is easy to handle, so the \dsm performs well while the \cdsm requires~$5\%$ to~$10\%$ more function evaluations.
For both algorithms, the contribution of the \search step is important at the beginning of the optimization process but declines when the algorithm approaches the solution.
Then, the \poll step starts performing instead.
The \covering step behaves similarly to the \search step.

\begin{figure}[!ht]
    \centering
    \includegraphics[width=0.9\linewidth]{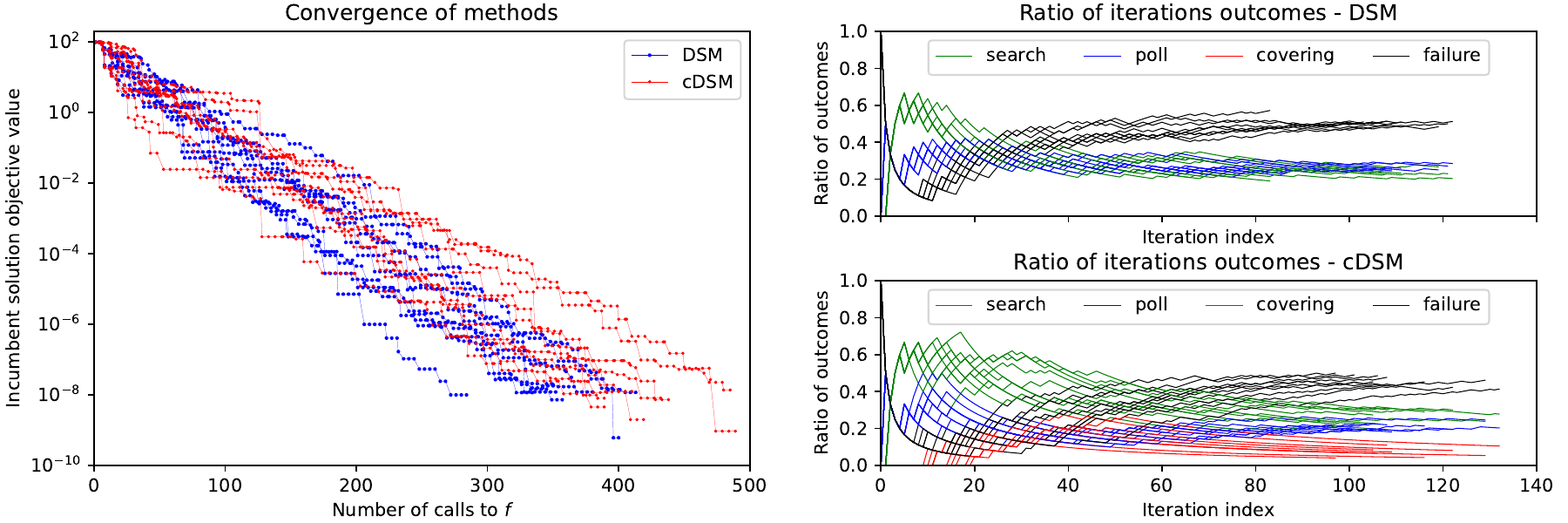}
    \caption{Results of our experiments on Problem~\eqref{problem:P_numerical_1}.}
    \label{figure:numerical_2}
\end{figure}

The results on Problem~\eqref{problem:P_numerical_2} are shown in Figure~\ref{figure:numerical_3}.
The \dsm converges close to the solution, but it fails to reliably find the continuity set of~$f$ containing it.
In~$4$ out of~$10$ cases, the \dsm does not leave the continuity set it started from.
In contrast, the \cdsm systematically converges to the solution from the correct continuity set.
Both algorithms trigger the stopping criterion~$k \geq 300$, but other tests conducted by relaxing this criterion show a similar outcome.
We suspect that this problem is difficult because the continuity set containing the solution does not have the ICP (from Definition~\ref{definition:ICP}).

\begin{figure}[!ht]
    \centering
    \includegraphics[width=0.9\linewidth]{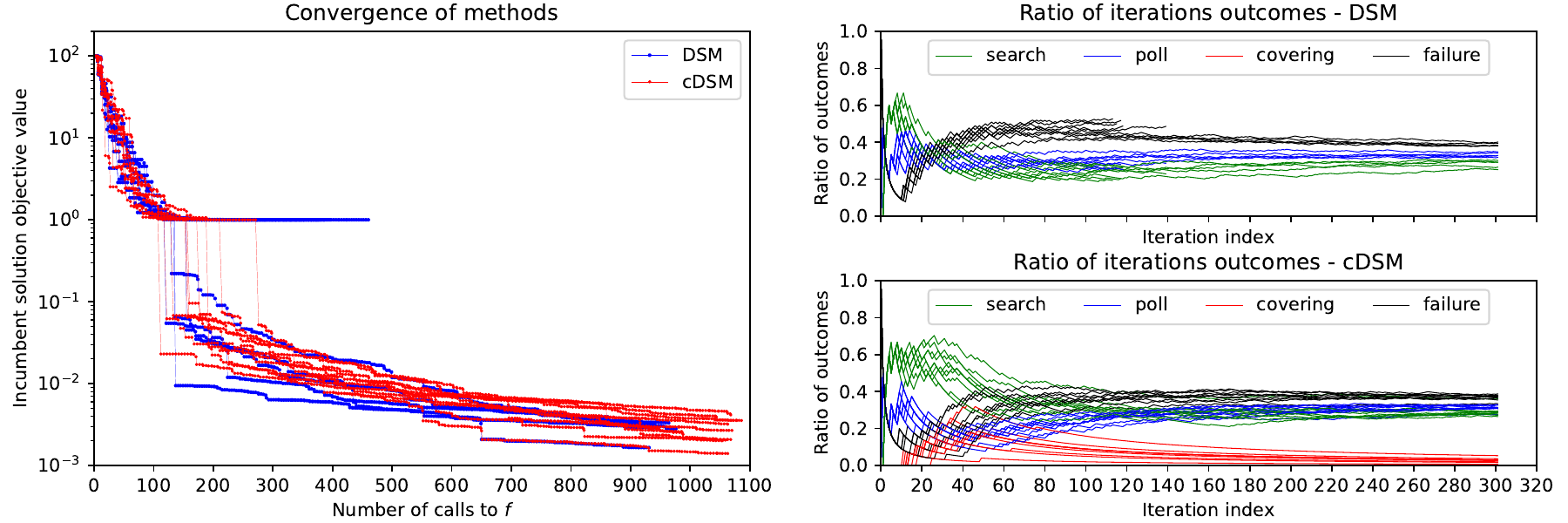}
    \caption{Results of our experiments on Problem~\eqref{problem:P_numerical_2}.}
    \label{figure:numerical_3}
\end{figure}

\section{Proofs of auxiliary results}
\label{appendix:propositions}

\begin{proposition}
    \label{proposition:not_ample_equiv_locally_thin}
    For all~$\Scal \subseteq \Rbb^n$,~$\Scal$ is locally thin if and only if it is not ample.
\end{proposition}
\begin{proof}
    Let~$\Scal \subseteq \Rbb^n$ be locally thin and let us prove by contradiction that it is not ample.
    Assume that~$\Scal$ is ample.
    Define~$\Ncal \subseteq \Rbb^n$ open such that~$\Scal \cap \Ncal \neq \emptyset = \int(\Scal) \cap \Ncal$, and take~$x \in \Scal \cap \Ncal$.
    Then~$x \in \cl(\int(\Scal)) \cap \Ncal$ and thus there exists~$\varepsilon > 0$ such that~$\Bcal_\varepsilon(x) \subseteq \Ncal$ and~$\Bcal_\varepsilon(x) \cap \int(\Scal) \neq \emptyset$.
    Hence~$\int(\Scal) \cap \Ncal \neq \emptyset$ which raises a contradiction.
    Reciprocally, let~$\Scal \subseteq \Rbb^n$ be not ample.
    Then,~$\Ncal \defequal \Rbb^n \setminus \cl(\int(\Scal))$ is open, and~$\Scal \cap \Ncal \neq \emptyset = \int(\Scal) \cap \Ncal$ by construction.
    Hence~$\Scal$ is locally thin.
\end{proof}

\begin{proposition}
    \label{proposition:ample_sets}
    We say that~$\Scal_1 \subseteq \Rbb^n$ \textit{has a dense intersection with}~$\Scal_2 \subseteq \Rbb^n$ if~$\Scal_2 \subseteq \cl (\Scal_1 \cap \Scal_2)$.
    The following properties hold for all~$\Scal_1 \subseteq \Rbb^n$,~$\Scal_2 \subseteq \Rbb^n$ and~$\Scal_3 \subseteq \Rbb^n$:
    \begin{enumerate}[label=\alph*)]
        \item
            \label{proposition:ample_sets:ample_inter_open}
            if~$\Scal_1$ is ample and~$\Scal_2$ is open, then~$\Scal_1 \cap \Scal_2$ is ample;
        \item
            \label{proposition:ample_sets:inner_dense_open_subset}
            if~$\Scal_1$ has a dense intersection with~$\Scal_2$ and~$\Scal_3$ is an ample subset of~$\Scal_2$, then~$\Scal_1$ has a dense intersection with~$\Scal_3$;
        \item
            \label{proposition:ample_sets:dense_in_ample_nonconstant_sequence}
            if~$\Scal_1$ has a dense intersection with~$\Scal_2$ and~$\Scal_2$ is ample, then, for all~$x \in \Scal_2$,~$\Scal_1 \setminus \{x\}$ has a dense intersection with~$\Scal_2$;
        \item
            \label{proposition:ample_sets:dense_in_ample_minus_point}
            if~$\Scal_1$ is open,~$\Scal_2$ is ample and~$\Scal_3$ has a dense intersection with~$\Scal_1$, then for all~$x \in \Scal_1$, the set~$\Scal_3 \setminus \{x\}$ has a dense intersection with~$\Scal_1 \cap \Scal_2$.
    \end{enumerate}
\end{proposition}
\begin{proof}
    Let us prove the first statement.
    Let~$\Scal_1 \subseteq \Rbb^n$ be ample and~$\Scal_2 \subseteq \Rbb^n$ be open, and let~$x \in \Scal_1 \cap \Scal_2$.
    We have~$x \in \Scal_1 \subseteq \cl(\int(\Scal_1))$, so there exists~$(x^k)_{k \in \Nbb}$ converging to~$x$ with~$x^k \in \int(\Scal_1)$ for all~$k \in \Nbb$.
    Since~$x^k \to x \in \Scal_2$ with~$\Scal_2$ open, it holds that~$x^k \in \Scal_2$ for all~$k$ large enough.
    Finally, we have~$x^k \in \int(\Scal_1) \cap \Scal_2 = \int(\Scal_1) \cap \int ( \Scal_2 ) = \int(\Scal_1 \cap \Scal_2)$ for all~$k$ large enough.
    We deduce that~$x \in \cl(\int(\Scal_1 \cap \Scal_2))$.
    
    Next, let us prove the second statement.
    Let~$\Scal_1 \subseteq \Rbb^n$ having a dense intersection with~$\Scal_2 \subseteq \Rbb^n$ and let~$\Scal_3 \subseteq \Scal_2$ be ample, and let~$x \in \Scal_3$.
    We have~$x \in \cl(\int(\Scal_3))$ so there exists~$(x^k)_{k \in \Nbb}$ converging to~$x$ with~$x^k \in \int(\Scal_3) \subseteq \Scal_2 \subseteq \cl(\Scal_1 \cap \Scal_2)$ for all~$k \in \Nbb$.
    Then, for all~$k \in \Nbb$, there exists~$(x^k_\ell)_{\ell \in \Nbb}$ converging to~$x^k$ with~$x^k_\ell \in \Scal_1 \cap \Scal_2 \cap \int(\Scal_3)$ for all~$\ell \in \Nbb$.
    For all~$k \in \Nbb$, let~$\ell(k) \in \Nbb$ be such that~$\textnorm{x^k_{\ell(k)}-x^k} \leq 2^{-k}$.
    It follows that~$(x^k_{\ell(k)})_{k \in \Nbb}$ converges to~$x$ and~$x^k_{\ell(k)} \in \Scal_1 \cap \Scal_2 \cap \int(\Scal_3) \subseteq \Scal_1 \cap \Scal_3$ for all~$k \in \Nbb$.
    Hence, we get that~$x \in \cl(\Scal_1 \cap \Scal_3)$.

    Now, let us prove the third statement.
    Let~$\Scal_1 \subseteq \Rbb^n$ and let~$\Scal_2 \subseteq \Rbb^n$ be ample.
    Assume that~$\Scal_1$ has a dense intersection with~$\Scal_2$ and let~$x \in \Scal_2$.
    Let~$y \in \Scal_2$.
    Then~$y \in \cl(\int(\Scal_2))$, so there exists~$(y^k)_{k \in \Nbb}$ converging to~$y$ with~$y \neq y^k \in \int(\Scal_2) \subseteq \Scal_2$ for all~$k \in \Nbb$.
    Let~$\varepsilon^k \defequal \textnorm{y^k-y} > 0$ for all~$k \in \Nbb$.
    Then~$(\varepsilon^k)_{k \in \Nbb}$ converges to~$0$.
    Now, remark that by dense intersection of~$\Scal_1$ with~$\Scal_2$, for all~$k \in \Nbb$ there exists~$z^k \in \Scal_1 \cap \Scal_2 \cap \Bcal_{\varepsilon^k}(y^k)$.
    Finally, let~$\kappa \defequal \min\{p \in \Nbb : \max_{k \geq p} \textnorm{y^k-y} \leq \frac{1}{2}\textnorm{y-x}\}$ if~$y \neq x$ and~$\kappa \defequal 0$ if~$y = x$.
    Hence, for all~$k \geq \kappa$ we have~$x \notin \Bcal_{\varepsilon^k}(y^k)$, and thus~$z^k \in \Scal_1 \cap \Scal_2 \cap \Bcal_{\varepsilon^k}(y^k) \setminus \{x\}$.
    Since~$(z^k)_{k \in \Nbb}$ converges to~$y$, it follows that~$y \in \cl((\Scal_1 \setminus \{x\}) \cap \Scal_2)$ as desired.

    Finally, let us prove the fourth statement.
    Let~$\Scal_1 \subseteq \Rbb^n$ open,~$\Scal_2$ ample and~$\Scal_3$ having a dense intersection with~$\Scal_1$.
    Then,~$\Scal_1 \cap \Scal_2$ is an ample subset of~$\Scal_1$ via Proposition~\refbis{proposition:ample_sets}{ample_inter_open}.
    Thus,~$\Scal_3$ has a dense intersection with~$\Scal_1 \cap \Scal_2$, from Proposition~\refbis{proposition:ample_sets}{inner_dense_open_subset} applied to~$\Scal_1 \defequal \Scal_3$ and~$\Scal_2 \defequal \Scal_1$ and~$\Scal_3 \defequal \Scal_1 \cap \Scal_2$.
    The claim follows from Proposition~\refbis{proposition:ample_sets}{dense_in_ample_nonconstant_sequence} applied to~$\Scal_1 \defequal \Scal_3$ and~$\Scal_2 \defequal \Scal_1 \cap \Scal_2$.
\end{proof}

\begin{proposition}
    \label{proposition:AuBaKo_implies_our_technical}
    Under Assumption~\ref{assumption:AuBaKo}, consider the family~$(X_i)_{i=1}^{N}$ it provides.
    For all~$i \in \ll1,N\rr$, denote by~$\cl_f(X_i) \defequal \{x \in \cl(X_i) : f_{|X_i \cup \{x\}}$ is continuous$\}$.
    Define~$I(x) \defequal \min \{i \in \ll1,N\rr : x \in \cl_f(X_i)\}$ for all~$x \in X$.
    Define~$Y_i \defequal \{x \in X : I(x) = i\}$ for all~$i \in \ll1,N\rr$.
    Then,~$X = \partitioncup_{i=1}^{N} Y_i$ and~$Y_i$ is an ample continuity set of~$f$ for all~$i \in \ll1,N\rr$.
\end{proposition}
\begin{proof}
    Consider the notation from Proposition~\ref{proposition:AuBaKo_implies_our_technical}.
    By design, the sets~$(Y_i)_{i=1}^{N}$ are pairwise disjoint and their union covers~$X$, so~$X = \partitioncup_{i=1}^{N} Y_i$.
    Then we prove that for all~$i \in \ll1,N\rr$,~$Y_i$ is ample and~$f_{|Y_i}$ is continuous.
    Let~$i \in \ll1,N\rr$.
    First, the properties of the sets~$(X_i)_{i=1}^{N}$ and the construction of~$I$ lead to
    \begin{equation*}
        \int(X_i)
        \overbracebelow{X_i ~\text{open}}{=}
        X_i
        \overbracebelow{I(X_i) = \{i\}}{\subseteq}
        Y_i
        \overbracebelow{
            x \notin \cl_f(X_i) \implies I(x) \neq i \\
            x \in \cl_f(X_i) \implies I(x) \leq i
            }{\subseteq}
        \cl_f(X_i) 
        \overbracebelow{\text{by construction}}{\subseteq}
        \cl(X_i).
    \end{equation*}
    Then,~$\int(Y_i) \supseteq \int(X_i)$ and~$Y_i \subseteq \cl(\int(X_i))$ so~$Y_i$ is ample.
    Moreover, let~$x \in Y_i$ and let~$(x^k)_{k \in \Nbb}$ converging to~$x$ with~$x^k \in Y_i$ for all~$k \in \Nbb$.
    For all~$k \in \Nbb$,~$f_{|X_i \cup \{x^k\}}$ is continuous so there exists~$(x^k_\ell)_{\ell \in \Nbb}$ converging to~$x^k$ such that~$x^k_\ell \in X_i$ for all~$\ell \in \Nbb$ and~$(f(x^k_\ell))_{\ell \in \Nbb}$ converges to~$f(x^k)$.
    Let~$\ell(k) \in \Nbb$ such that~$\textabs{f(x^k_{\ell(k)})-f(x^k)} \leq 2^{-k}$ and~$\textnorm{x^k_{\ell(k)}-x^k} \leq 2^{-k}$.
    Then~$(x^k_{\ell(k)})_{k \in \Nbb}$ converges to~$x$ and~$x^k_{\ell(k)} \in X_i$ for all~$k \in \Nbb$, so~$(f(x^k_{\ell(k)}))_{k \in \Nbb}$ converges to~$f(x)$ by continuity of~$f_{|X_i \cup \{x\}}$.
    Hence~$(f(x^k))_{k \in \Nbb}$ converges to~$f(x)$.
    Thus~$f_{|Y_i}$ is continuous at~$x$, as desired.
\end{proof}

\begin{proposition}
    \label{proposition:new_ECP_equiv_old}
    A set has the interior cone property from Definition~\ref{definition:ICP} if and only if its complement has the exterior cone property from~\cite[Definition~4.1]{ViCu2012} (quoted below).
    \begin{quote}\textbf{\cite[Definition~4.1]{ViCu2012}}
        \label{definition:ECP_old}
        A set~$\Scal \subseteq \Rbb^n$ is said to \textit{have the exterior cone property} if at all points~$x \in \partial \Scal$ there exists a cone~$\Kcal_x \defequal \pointplusset{x}{\Rbb_+^*\Ucal}$ (with $\emptyset \neq \Ucal \subseteq \Sbb^n$ open in induced topology) emanating from~$x$, a neighborhood~$\Ocal_x$ of~$x$ and an angle~$\theta > 0$ such that~$\Ecal_x \subseteq \Scal^c$ and~$\Theta(e-x,a-x) \geq \theta$ for all~$(e,a) \in \Ecal_x \times \Scal_x$, where~$\Theta(\cdot,\cdot)$ computes the unsigned internal angle between two vectors and~$\Scal^c \defequal (\Rbb^n \setminus \Scal)$ and~$\Ecal_x \defequal (\Kcal_x \cap \Ocal_x)$ and~$\Scal_x \defequal (\Scal \cap \Ocal_x \setminus \{x\})$.
    \end{quote}
\end{proposition}
\begin{proof}
    Denote by~$\mathrm{(ICP)}$ and~$\mathrm{(ECP)}$ the interior and exterior cone properties stated in Definition~\ref{definition:ICP}, and by~$\mathrm{[ECP]}$ the exterior cone property stated in~\cite[Definition~4.1]{ViCu2012}.
    A set has the~$\mathrm{(ICP)}$ if and only if its complement has the~$\mathrm{(ECP)}$, so we only need to prove that the~$\mathrm{(ECP)}$ is equivalent to the~$\mathrm{[ECP]}$.
    Denote by~$[x,y] \defequal \{x+t(y-x) : t \in [0,1]\}$, for all~$(x,y) \in (\Rbb^n)^2$.
    Let~$\Scal \subseteq \Rbb^n$ and~$\Scal^c \defequal (\Rbb^n \setminus \Scal)$.
    
    Assume that~$\Scal$ has the~$\mathrm{[ECP]}$.
    For all~$x \in \partial\Scal^c = \partial\Scal$, the~$\mathrm{[ECP]}$ of~$\Scal$ applied to~$x$ provides the cone~$\Kcal_x$ and the neighborhood~$\Ocal_x$ such that~$(\Kcal_x \cap \Ocal_x) \subseteq \Scal^c$.
    Then~$\Scal^c$ has the~$\mathrm{(ICP)}$, so~$\Scal$ has the~$\mathrm{(ECP)}$.
    Thus the~$\mathrm{[ECP]}$ implies the~$\mathrm{(ECP)}$.
    
    Assume that~$\Scal$ has the~$\mathrm{(ECP)}$.
    Let~$x \in \partial\Scal$.
    The~$\mathrm{(ECP)}$ of~$\Scal$ applied to~$x$ provides~$\emptyset \neq \Ucal \subsetneq \Sbb^n$ open in~$\Sbb^n$ and~$\Kcal \defequal \Rbb_+^*\Ucal$ and the neighborhood~$\Ocal$ of~$0$ such that~$(\Kcal_x \cap \Ocal_x) \subseteq \Scal^c$, where~$\Kcal_x \defequal \pointplusset{x}{\Kcal}$ and~$\Ocal_x \defequal \pointplusset{x}{\Ocal}$.
    Let~$\theta > 0$ small enough so that the open set~$\Ucal' \defequal \{u \in \Ucal : \Theta(u,v) > \theta,~ \forall v \in \Sbb^n \setminus \Ucal\}$ is nonempty.
    Define~$\Kcal' \defequal \Rbb_+^*\Ucal'$ and~$\Kcal'_x \defequal \pointplusset{x}{\Kcal'} \subseteq \Kcal_x$ and~$\Ecal_x \defequal (\Kcal'_x \cap \Ocal_x)$ and~$\Scal_x \defequal (\Scal \cap \Ocal_x \setminus \{x\})$.
    Then~$\Ecal_x \subseteq \Scal^c$ and~$\theta$ satisfies the requirement in the~$\mathrm{[ECP]}$.
    Indeed,~$\Ecal_x \subseteq \Kcal_x$ while~$\Scal_x \cap \Kcal_x = \emptyset$, thus for all~$(e,a) \in \Ecal_x \times \Scal_x$ there exists~$y \in \partial\Kcal_x \cap [e,a] \neq \emptyset$.
    Since~$e$ and~$y$ and~$a$ belong to the same line, we get~$\Theta(e-x,a-x) = \Theta(e-x,y-x) + \Theta(y-x,a-x) = \Theta(\frac{e-x}{\textnorm{e-x}},\frac{y-x}{\textnorm{y-x}}) + \Theta(y-x,a-x) \geq \theta + 0$.
    The first term is greater than~$\theta$ since~$\frac{e-x}{\textnorm{e-x}} \in \Ucal'$ while~$\frac{y-x}{\textnorm{y-x}} \in \Sbb^n \setminus \Ucal$, and the second term is positive by definition of~$\Theta$.
    Thus,~$\Scal$ has the~$\mathrm{[ECP]}$ at~$x$.
    Hence the~$\mathrm{(ECP)}$ implies the~$\mathrm{[ECP]}$.
\end{proof}

\begin{proposition}
    \label{proposition:AuBoBo_implies_our_technical}
    If~$\Scal \subseteq \Rbb^n$ has the ICP from Definition~\ref{definition:ICP}, then~$\Scal$ is ample.
    The reciprocal is not true.
\end{proposition}
\begin{proof}
    Let~$\Scal \subseteq \Rbb^n$ having the ICP.
    Let~$x \in \Scal$.
    The ICP provides~$\Kcal_x$ and~$\Ocal_x$ satisfying~$(\Kcal_x \cap \Ocal_x) \subseteq \Scal$ and~$x \in \cl(\Kcal_x \cap \Ocal_x)$.
    Moreover,~$\Kcal$ is open as the image of~$\Rbb_+^* \times \Ucal$ (an open subset of~$\Rbb_+^* \times \Sbb^n$) by the homeomorphism~$(\lambda,u) \in \Rbb_+^* \times \Sbb^n \mapsto \lambda u \in \Rbb^n \setminus \{0\}$.
    Thus,~$\Kcal_x \cap \Ocal_x$ is also open, so~$(\Kcal_x \cap \Ocal_x) \subseteq \int(\Scal)$.
    Hence,~$x \in \cl(\int(\Scal))$, which proves the direct implication.
    To observe that the reciprocal implication fails, consider~$n \defequal 2$ and~$\Scal \defequal \mathrm{epi}(\sqrt{\textabs{\cdot}})$.
    Then~$\Scal$ is ample but the ICP fails at~$x = (0,0) \in \cl(\Scal)$.
\end{proof}

\clearpage\newpage
\bibliography{bibliography.bib}

\begin{thebibliography}{10}

\bibitem{AbAuDeLe09}
M.A. Abramson, C.~Audet, J.E. {Dennis, Jr.}, and S.~{Le~Digabel}.
\newblock {OrthoMADS: A Deterministic MADS Instance with Orthogonal
  Directions}.
\newblock {\em SIAM Journal on Optimization}, 20(2):948--966, 2009.

\bibitem{AuBaKo22}
C.~Audet, A.~Batailly, and S.~Kojtych.
\newblock Escaping unknown discontinuous regions in blackbox optimization.
\newblock {\em SIAM Journal on Optimization}, 32(3):1843--1870, 2022.

\bibitem{AuBoBo22}
C.~Audet, P.-Y. Bouchet, and L.~Bourdin.
\newblock Counterexample and an additional revealing poll step for a result of
  “analysis of direct searches for discontinuous functions”.
\newblock {\em Mathematical Programming}, 2024.

\bibitem{AuDe03a}
C.~Audet and J.E. {Dennis, Jr.}
\newblock Analysis of generalized pattern searches.
\newblock {\em SIAM Journal on Optimization}, 13(3):889--903, 2003.

\bibitem{AuDe2006}
C.~Audet and J.E. {Dennis, Jr.}
\newblock {Mesh Adaptive Direct Search Algorithms for Constrained
  Optimization}.
\newblock {\em SIAM Journal on Optimization}, 17(1):188--217, 2006.

\bibitem{AuDe09a}
C.~Audet and J.E. {Dennis, Jr.}
\newblock {A Progressive Barrier for Derivative-Free Nonlinear Programming}.
\newblock {\em SIAM Journal on Optimization}, 20(1):445--472, 2009.

\bibitem{AuHa2017}
C.~Audet and W.~Hare.
\newblock {\em Derivative-Free and Blackbox Optimization}.
\newblock Springer Series in Operations Research and Financial Engineering.
  Springer, Cham, Switzerland, 2017.

\bibitem{AuHa20}
C.~Audet and W.~Hare.
\newblock Model-based methods in derivative-free nonsmooth optimization.
\newblock In {\em Numerical nonsmooth optimization}, chapter~15. Springer,
  2020.

\bibitem{BaUs21OneDimDisc}
K.~Barkalov and M.~Usova.
\newblock A search algorithm for the global extremum of a discontinuous
  function.
\newblock In {\em Advances in Optimization and Applications}, pages 37--49,
  Cham, 2021. Springer.

\bibitem{BeCaSc21}
A.S. Berahas, L.~Cao, and K.~Scheinberg.
\newblock Global convergence rate analysis of a generic line search algorithm
  with noise.
\newblock {\em SIAM Journal on Optimization}, 31(2):1489--1518, 2021.

\bibitem{BouchetPhD}
P.-Y. Bouchet.
\newblock {\em Théorie de l'optimisation sans dérivées dans le cas
  discontinu}.
\newblock PhD thesis, Polytechnique Montréal, december 2023.

\bibitem{CoMo98PWDisc}
A.R. Conn and M.~Mongeau.
\newblock Discontinuous piecewise linear optimization.
\newblock {\em Mathematical programming}, 80(3):315--380, 1998.

\bibitem{CoScVibook}
A.R. Conn, K.~Scheinberg, and L.N. Vicente.
\newblock {\em Introduction to Derivative-Free Optimization}.
\newblock MOS-SIAM Series on Optimization. SIAM, Philadelphia, 2009.

\bibitem{Ha60}
J.H. Halton.
\newblock On the efficiency of certain quasi-random sequences of points in
  evaluating multi-dimensional integrals.
\newblock {\em Numerische Mathematik}, 2(1):84--90, 1960.

\bibitem{KvSe20GLobalOptim}
D.E. Kvasov and Y.D. Sergeyev.
\newblock {\em Lipschitz Expensive Global Optimization}, pages 1--18.
\newblock Springer International Publishing, Cham, 2020.

\bibitem{LaMeWi2019}
J.~Larson, M.~Menickelly, and S.M. Wild.
\newblock {Derivative-free optimization methods}.
\newblock {\em Acta Numerica}, 28:287--404, 2019.

\bibitem{LedWild2015}
S.~{Le~Digabel} and S.M. Wild.
\newblock {A taxonomy of constraints in black-box simulation-based
  optimization}.
\newblock {\em Optimization and Engineering}, 25(2):1125--1143, 2024.

\bibitem{Le63SemiOpen}
N.~Levine.
\newblock Semi-open sets and semi-continuity in topological spaces.
\newblock {\em American Mathematical Monthly}, 70:36--41, 1963.

\bibitem{Meijster2000}
A.~Meijster, J.B.T.N. Roerdink, and W.H. Hesselink.
\newblock {\em A General Algorithm for Computing Distance Transforms in Linear
  Time}, pages 331--340.
\newblock Springer US, Boston, MA, 2000.

\bibitem{RoRo23ReducedSpaces}
L.~Roberts and C.W. Royer.
\newblock Direct search based on probabilistic descent in reduced spaces.
\newblock {\em SIAM Journal on Optimization}, 33(4):3057--3082, 2023.

\bibitem{Rock80a}
R.T. Rockafellar.
\newblock Generalized directional derivatives and subgradients of nonconvex
  functions.
\newblock {\em Canadian Journal of Mathematics. Journal Canadien de
  Math\'ematiques}, 32(2):257--280, 1980.

\bibitem{SeKv17GlobalOptim}
Y.D. Sergeyev and D.E. Kvasov.
\newblock {\em Deterministic Global Optimization}.
\newblock Springer Briefs in Optimization. Springer New York, 2017.

\bibitem{StPa23BookDIRECT}
L.~Stripinis and R.~Paulavičius.
\newblock {\em Derivative-free DIRECT-type Global Optimization}.
\newblock Springer Briefs in Optimization. Springer, Cham, Switzerland, 2023.

\bibitem{StSe13GlobalOptim}
R.G. Strongin and Y.D. Sergeyev.
\newblock {\em Global Optimization with Non-Convex Constraints}.
\newblock Nonconvex Optimization and Its Applications. Springer, New York,
  2013.

\bibitem{ViCu2012}
L.N. Vicente and A.L. Cust\'odio.
\newblock Analysis of direct searches for discontinuous functions.
\newblock {\em Mathematical Programming}, 133(1-2):299--325, 2012.

\bibitem{WeBa14GlobalDisc}
A.~Wechsung and P.I. Barton.
\newblock Global optimization of bounded factorable functions with
  discontinuities.
\newblock {\em Journal of Global Optimization}, 58:1--30, 2014.

\bibitem{ZhXi20}
D.~Zhan and H.~Xing.
\newblock Expected improvement for expensive optimization: a review.
\newblock {\em Journal of Global Optimization}, 78(3):507--544, 2020.

\bibitem{ZhZi21GlobalBayesian}
A.~Zhigljavsky and A.~Žilinskas.
\newblock {\em Bayesian and High-Dimensional Global Optimization}.
\newblock Springer Briefs in Optimization. Springer, Cham, 2021.

\end{thebibliography}

\end{document}